\documentclass{amsart}[10pt]

\usepackage[T1]{fontenc}
\usepackage{amssymb, amsmath}
\usepackage{amssymb, amsthm, amsmath}
\usepackage{array}

\usepackage{amstext, units, mathtools}

\usepackage{comment}

\usepackage{url}
\usepackage[latin9]{inputenc}
\usepackage[all]{xy}
\usepackage[colorlinks=true]{hyperref}

\DeclareMathOperator{\aut}{Aut}

\DeclareMathOperator{\reg}{reg}

\DeclareMathOperator{\cl}{Cl}

\theoremstyle{plain}

\newtheorem{theorem}{Theorem}[section]

\newtheorem{claim}[theorem]{Claim}

\newtheorem{corollary}[theorem]{Corollary}

\newtheorem{fact}[theorem]{Fact}
\newtheorem{lemma}[theorem]{Lemma}
\newtheorem{proposition}[theorem]{Proposition}

\newtheorem{summary}[theorem]{Summary}

\theoremstyle{definition}
\newtheorem{definition}[theorem]{Definition}

\theoremstyle{remark}
 \newtheorem{remark}[theorem]{Remark}

\newcommand{\Rr}{{\mathds{R}}}

\newcommand{\x}{\bar x}

\renewcommand{\phi}{\varphi}

\renewcommand{\x}{\bar x}

\long\def\symbolfootnote[#1]#2{\begingroup%
\def\thefootnote{\fnsymbol{footnote}}\footnote[#1]{#2}\endgroup}

\makeatletter

\def\Ind#1#2{#1\setbox0=\hbox{$#1x$}\kern\wd0\hbox to 0pt{\hss$#1\mid$\hss}
\lower.9\ht0\hbox to 0pt{\hss$#1\smile$\hss}\kern\wd0}

\def\Notind#1#2{#1\setbox0=\hbox{$#1x$}\kern\wd0\hbox to 0pt{\mathchardef
\nn=12854\hss$#1\nn$\kern1.4\wd0\hss}\hbox to
0pt{\hss$#1\mid$\hss}\lower.9\ht0 \hbox to
0pt{\hss$#1\smile$\hss}\kern\wd0}

\def\dim{{\rm dim}}
\def\gap{{\rm gap}}
\def\reg{{\rm reg}}
\def\cl{{\rm cl}}
\def\sq{{\rm Sq}}

\def\Rr{\mathbb R}
\def\Rc{\mathcal R}

\def\EE{\langle E_1,\dots, E_n\rangle}
\def\EEp{\langle E_1',\dots, E_n'\rangle}
\def\EEn-1{\langle E_1,\dots, E_{n-1}\rangle}



\begin{document}

\title{Decomposable Ordered Groups}

\author{Eliana Barriga}
\address{Eliand Barriga  \\
Departamento de Matem\'aticas \\
Universidad de los Andes \\
Cra 1 No. 18A-10, Edificio H \\
Bogot\'a, 111711 \\
Colombia}

\email{el.barriga44@uniandes.edu.co}

\author{Alf Onshuus}
\address{Alf Onshuus \\
Departamento de Matem\'aticas \\
Universidad de los Andes \\
Cra 1 No. 18A-10, Edificio H \\
Bogot\'a, 111711 \\
Colombia}

\email{aonshuus@uniandes.edu.co}
\urladdr{http://matematicas.uniandes.edu.co/aonshuus}

\author{Charles Steinhorn}
\address{ Charles Steinhorn\\ Vassar College \\Poughkeepsie, NY
USA}
\email{steinhorn@vassar.edu}

\thanks{The third author was partially supported by NSF grant DMS-0801256.}

\maketitle


\section{Introduction and Preliminaries}

Definable groups in o-minimal structures have been studied
by many authors for well over twenty years. This context includes all real algebraic groups,
complex algebraic groups, and compact real Lie groups. One of the most striking results in this
area (see \cite{HPP}) reveals a tight correspondence between
definably compact groups in o-minimal expansions of real closed fields and compact real Lie groups.
In \cite{OnSt} the second and third authors generalized the concept of o-minimality
to include `higher dimensional' ordered structures.

It was soon apparent that the results in \cite{OnSt} have consequences
for ordered groups in this new context that would imply structure results for ordered real Lie groups that might not be known to researchers in that area. This paper, which we outline in the next paragraphs, develops these ideas and in particular undertakes an analysis of `low-dimensional' ordered groups definable in a sufficiently rich o-minimal expansion of a real closed field.


In Section \ref{section2} we work
in the general framework introduced in \cite{OnSt}, {\em decomposable ordered structures\/} (Definition~\ref{amenable}), under the additional assumption that a group operation is defined on the structure. We call such a structure a ``decomposable ordered group."\footnote{It should be pointed out is not to be confused with the use of the term decomposable group that has long appeared in the literature (see~\cite{Fuchs}, e.g.). Since our terminology appears only in Section~\ref{section2}, we ask the reader to forgive any possible confusion.} The main result at this
level of generality is Theorem~\ref{decomposable}, which asserts that any such group is supersolvable, and that
topologically it is homeomorphic to the product of o-minimal groups.

To proceed further to classify ordered groups in this context, we must work out the possible group extensions of a triple of a group, module, and action: while we can prove the existence of a convex normal o-minimal subgroup (Theorem \ref{o-minimal subgroup}),
in order to use this inductively to determine the structure of the given ordered group, we need to understand possible extensions. This forces us to restrict our focus beginning in Section~\ref{section3} to ordered groups that are definable in o-minimal structures, and to develop the connection between group extensions and the second cohomology group of the o-minimal group cohomology (Theorem \ref{2-Cohomology}).

With this material in hand, we make two additional assumptions beginning in Sections~\ref{subsection3.2} and~\ref{dimension2} to make our analysis feasible. The first is that we work in an o-minimal field
$\Rc$ in which all o-minimal definable subgroups of $\Rc$ are definably isomorphic.
While this may seem at first to be a very strict assumption, as discussed in~\ref{subsection3.2} it is quite natural in our context.
We make our second additional assumption in Section~\ref{dimension2}, namely that in
$\mathcal{R}$ every abelian definable extension of $(R,+)$ by $(R,+)$ splits. The issue of whether such an extension splits is delicate; whereas it is not difficult to show this for a nonabelian extension, in general we do not know if it splits in the abelian case (see \cite{PS05}, e.g.)

For our analysis of ordered groups of dimension~2 and~3 in our setting, we first determine these groups modulo definable group isomorphism. This is undertaken in Section~\ref{dimension2} for dimension~2 groups, culminating in Theorem~\ref{T:1.3}. For the dimension~3 analysis, facts about group cohomology employing spectral sequence methods come into play; these appear in the first author's M.Sc. thesis \cite{Ba} written under the supervision of the second author and are presented here in Section~\ref{spectral}. The dimension~3 analysis, carried out in Section~\ref{dimension3}, divides into four cases, each treated in a separate subsection. The analysis in Sections~\ref{dimension2}-\ref{dimension3} is extended in Section~\ref{sectionOrdered}
to determine the groups up to definable {\em ordered} group isomorphism.

The principal results obtained in this paper are collected in Section~\ref{Summary} in Summaries~\ref{T:3*}-\ref{3-dimensional}.
All the results in this paper apply to ordered groups that are definable
in the Pfaffian closure of the real field, and all the groups that appear
in our classification are real Lie groups. The connections with real Lie groups is discussed in Section~\ref{Summary} as well.

For a reader not interested in o-minimal structures per se---or, for that matter, our
generalization---the natural question is to what extent our work gives a complete characterization of
ordered real Lie groups under the assumption that the order is itself smooth (or analytic).
First, the context in which we work \emph{does} include all real semi-algebraic ordered Lie groups (where the order is also defined by semi-algebraic equations). In the general case, it is known that any compact Lie group is isomorphic
to a real algebraic subgroup of $GL(n, \mathbb R)$ (see, \cite{Kn}, e.g.), so is definable in the real field
and thus analyzable by o-minimal technology. While it also is
not difficult to find examples of Lie groups (both over the complex and real fields) that are not definable in
an o-minimal structure, the ordered group context rules out all the examples known to us: to find an
ordered real Lie group that does not fall under our assumptions it would be necessary to find a torsion free group that is diffeomorphic to
$\mathbb R^n$ with both the order relation and the group operation given by analytic functions, but which is not definable in the Pfaffian closure of the real field.

With our synopsis of the paper complete, we conclude this section by recording several definitions and results from \cite{OnSt} that we require.

\begin{definition}
Given a structure $\mathcal C$, we say that a pair $(X,<_X)$ consisting of a definable subset $X$ and
a definable linear order $<_X$ is \emph{o-minimal} if every $\mathcal C$-definable subset of $X$
is a finite union of (open) $<_X$-intervals and points.
\end{definition}

\begin{definition}\label{amenable}

A (definable) linearly ordered set  $(X,<)$ in an ambient structure $\mathcal C$ is \emph{decomposable}
if there are ($\mathcal C$-) definable equivalence relations $E_1,\dots, E_n$ such
that
\begin{itemize}
\item[i.] $E_1$ is an equivalence relation on $X$ with finitely
many classes;

\item[ii.]  For $i=1,\ldots ,n-1$, we have that $E_{i+1}$ is an
equivalence relation on $X$ refining $E_i$;

\item[iii]  For each $i<n$ there is a definable linear order $<_i$
on the set of $E_{i+1}$-classes such that for all $w\in X$ the set

\[\left\{ [x]_{E_{i+1}} \left|\right. E_{i}(x,w) \right\}\]
ordered by $<_i$ is a dense o-minimal set.

\item [iv.] For all $w\in X$, the set $[x]_{E_{n}}$ ordered by $<$
(the order inherited from $X$) is dense o-minimal.
\end{itemize}

If such a sequence of equivalence relations exist, we say that the
sequence $\langle E_1, \dots, E_n\rangle$ is a \emph{presentation}
of $X$. The empty sequence is a presentation of a finite set.

\bigskip

If $\EE$ is a presentation of a decomposable set $(X,<)$, we
define the \emph{depth of the presentation} to be $n$, the number of
equivalence relations in the presentation.
\end{definition}

Some presentations are particularly easy to work
with:

\begin{definition}\label{def:minimal}

Let $(X,<)$ be a linearly ordered set in an ambient structure $\mathcal C$.

\begin{itemize}
\item[i.] $(X,<)$ is a \emph{0-minimal set} (has a \emph{minimal presentation
of depth $0$}) if it is a singleton and its presentation is the empty
tuple.

\item[ii.] $(X,<)$ is an \emph{$(n+1)$-minimal set} (has a \emph{minimal presentation
of depth $(n+1)$}) if there is a family
$\{I_b\}_{b\in B}$ of pairwise disjoint \emph{infinite} dense
o-minimal sets such that $X=\bigcup_{b\in B} I_b$ and there is a
definable order $<_B$ so that $(B,<_B)$ is an $n$-minimal set.
\end{itemize}
\end{definition}


If a set has a minimal presentation, its dimension can be
easily defined, as we now observe. We note that it is not difficult to show that the definition
coincides with topological dimension for definable sets in o-minimal theories.

\begin{definition}\label{lexdimension}
Let $(X,<)$ be decomposable in an ambient structure $\mathcal C$. A \emph{minimal
decomposition of $X$} is a partition of $X$ into finitely many
minimal sets.

We say that $X$ has  \emph{dimension $n$}, denoted by $\dim(X)=n$,
if $X$ has a minimal decomposition into minimal sets whose maximal
depth is $n$ and at least one of the minimal sets in the
decomposition is an $n$-minimal set. The \emph{degree} of $X$ is
$m$ if $X$ has dimension $n$ and $m$ is least such that there is a
minimal decomposition of $X$ which contains $m$ distinct
$n$-minimal sets.
\end{definition}

Fact \ref{minimal}, below, provides that a decomposable set has a minimal decomposition.

We require one additional definition to state Fact~\ref{main theorem},
the main result in \cite{OnSt}.

\begin{definition} \label{lexpresentdef}
A presentation $\EE$ of a definable set $(X,<)$ in an ambient structure $\mathcal C$
is called a
\emph{lexicographic presentation} if for every $x\in X$ and every
$i\leq n$ the class $[x]_{E_i}$ is convex with respect to $<$ and
for each $i\leq n$ the order $(X/E_i,<_i)$ given by the definition
of a presentation is the order inherited from $(X,<)$ (by
convexity the order is unambiguosly defined by any of the
representatives of the class).
\end{definition}

\begin{fact}[Theorem 5.1) in \cite{OnSt}]\label{main theorem}
Every decomposable set $X$ can be partitioned into finitely many
subsets, each of which has a lexicographic presentation.
\end{fact}

We also shall avail ourselves of the next fact; it is Lemma~5.2 in
\cite{OnSt}.

\begin{fact}\label{minimal}
Let $\{X_j\}_{j\in J}$ be a (definable) family of decomposable
subsets in a decomposable set $X$  with uniformly definable
non-redundant presentations of dimension $n$. Then there are a
positive integer $k$ and, for each $j\in J$, sets $Y^j_1,\ldots
,Y^j_k$ that partition $X_j$ such that:
\begin{itemize}
\item[i.] for each $i\leq k$ the
family $\{Y^j_i\}_{j\in J}$ is a definable family of minimal sets
whose presentations are given uniformly definably;

\item[ii.] for each $j\in J$ and $i\leq k$ for which $Y^j_i$  is
$n\/$-minimal, the ordering given by its presentation is the
restriction of  the ordering on $X_j$ given by its presentation.
\end{itemize}
In particular, if the family is just the singleton $\{X\}$ and the
presentation of $X$ is lexicographic, then so are the
presentations of the $n$-minimal sets.
\end{fact}

\begin{corollary}\label{cor:minimal}
Let $\EE$ be a presentation of $X$. Then there is a presentation
$\EEp$ of $X$ such that the restriction of $\EEp$ to any of the
the $E_1'$-classes of dimension $n$ is a minimal lexicographic presentation.
\end{corollary}

\medskip

Throughout the paper we will work in the category of definable groups in some o-minimal 
expansion of a real closed field. So all groups, homomorphisms and actions will be assumed 
to be definable.

\section{Existence of a normal convex o-minimal subgroup}\label{section2}

We here prove Theorem \ref{o-minimal subgroup}, which asserts that
every decomposable ordered group $G$ has an o-minimal convex normal
subgroup. This implies (Theorem~\ref{decomposable}) that every such
group is solvable and as {\em an ordered set} is isomorphic to a
finite lexicographic product of o-minimal groups. In particular, if
the group $G$ is definable in an o-minimal structure $\mathcal M$ in
which every two definable o-minimal groups are (definably)
isomorphic---at least those that appear as quotients in the normal
chain witnessing that $G$ is solvable---it follows that $G$ as an
ordered set is isomorphic to $M^n$ ordered lexicographically. In
fact, with the subset topology this isomorphism is a homeomorphism.
Although the context in which all definable o-minimal groups are
isomorphic may seem overly restrictive, it is not so: we say
more about this in Section~\ref{subsection3.2}.

Throughout this section $G$ denotes a decomposable ordered group.
The proof of Theorem \ref{o-minimal subgroup} proceeds as follows.
We first apply Fact~\ref{main theorem}
and Corollary~\ref{cor:minimal} to give adequate presentations of
$G$ and state some properties of the presentations we obtain in our
context. We then prove the existence of a presentation in which
each of the finest equivalence classes is an o-minimal convex
interval. This implies Theorem \ref{o-minimal subgroup} via an
analysis of how the group structure interacts with this
presentation.

\bigskip

We begin by some easy consequences of the results in
\cite{OnSt} applied to groups.

\begin{proposition}\label{no successor} An ordered decomposable group $G$
is dense with respect to its order $<$.
That is, no element in $G$ has an
immediate successor under $<$.
\end{proposition}

\begin{proof}
By group translation, if one element has an immediate successor
then all elements do. This implies the existence of
infinite discretely ordered chains, which is impossible in a decomposable
structure, as no ordered theory with an infinite
discrete chain satisfies uniform finiteness.
\end{proof}

Let $\EE$ be a presentation of $G$ as provided by Corollary
\ref{cor:minimal}.  Additionally let $C_1,\dots,C_k, C_{k+1}, \dots, C_t$ be an
enumeration of the finitely many $E_1$-classes such that $G=C_1\cup \dots \cup C_t$,
gives a decomposition of $X$ into sets with
minimal lexicographic presentations such that $\dim(C_i)=\dim (G)=n$ for $1\leq i\leq k$.
For the remainder of this section, we assume that $G$ is presented this way.

As in \cite{OnSt}, for $1\leq i\leq t$ and $x\in C_i$ we define the {\em gap\/} of $x$ in $C_i$
with respect to $G$, denoted $\gap (x, C_i, G)$, by
\smallskip
\[
\begin{split}
\{x\}\,\cup \,&\{ y\in G \mid y>x \,\wedge\, (\forall z\in C_i)\, z>x\rightarrow z>y \}\\
\cup\, &\{ y\in G \mid y<x \,\wedge\, (\forall z\in C_i)\, z<x\rightarrow
z<y\}.
\end{split}
\]
It is easy to see that if $\gap(x, C_i, G)$ is finite but not a
singleton, then some element of $\gap (x, C_i, G)$ is an immediate
successor of $x$. This yields a direct corollary of~\ref{no successor}.

\begin{proposition}\label{singleton}
For all $x\in C_i$ the set $\gap(x, C_i, G)$ is
either a singleton or infinite.
\end{proposition}

Let $1\leq i\leq k$. Observe that additivity of dimension
implies that the set of all $x\in C_i$ for which $\gap(x, C_i, G)$
is infinite has dimension at most $n-1$.
By  $\reg(C_i)$ we denote the set of all $x\in C_i$ such that $\gap(x, C_i, G)$ is a singleton. Thus,
if $E_1,\dots, E_n$ is a minimal lexicographic decomposition for
$C_i$, then the set of $E_n$-classes that contain an interval in $\reg(C_i)$
has dimension $n-1$; even more, the set of such $E_n$-classes has
codimension less than $n-1$. Since the $E_n$-classes are o-minimal and
convex in $C_i$, it follows that
\[
\gap(x, C_i, G)=\gap(x, [x]_{E_n}, G).
\]

\begin{proposition}\label{singletongap}
Let $I$ be an o-minimal subset of $G$. Then $\gap(a, I, G)=\{a\}$ for all $a\in I$ if and only if $I$ is convex in $G$.
\end{proposition}

\begin{proof}
Only the left-to-right implication requires any argument. By o-minimality of $I$
there is some $a\in I$ for which, without loss of generality, there is some
$b\in G\setminus I$ such that $x<b<a$ for all $x\in I$ with $x<a$.
Thus $b\in \gap(a,I, G)$, as required.
\end{proof}

The following is an immediate consequence of the proposition above
and the discussion preceding it.

\begin{lemma}
Let $1\leq i\leq k$ and $I$ be an interval contained in an $E_n$-class in $C_i$ such
that $I\subset \reg(C_i)$. Then $I$ is convex in $G$.
In particular, the set of all $x\in C_i$
contained in such an interval
has dimension $n$ (in fact, codimension less than $n$).
\end{lemma}

\begin{proof}
Let $x\in I$. Since $I$ is a convex subset of $[x]_{E_n}$
which, since the presentation of $C_i$ is lexicographic,
is a convex subset of $C_i$, we have
\[
\gap(x, I, G)=\gap(x, C_i, G).
\]
As $I\subseteq \reg(C_i)$ it follows by~\ref{singletongap} that
$I$ is convex in $G$.
\end{proof}

We now denote by $\reg(G)$ the set of all $x\in G$ such
that there is a convex o-minimal $I_x\subset G$ with $x\in I_x$.

\begin{lemma}\label{limit}
Let $I$ be convex in $G$ and $x\in I$. Then there is
some $j$ and some o-minimal convex $I'\subset C_j\cap \reg(G)$
such that $\{x\}\cup I'$ is convex in $G$. If $x\not\in I'$ then
$I'$ can be chosen so that $x<I'$ and $x$ is an infimum for $I'$ in $G$.
\end{lemma}

\begin{proof}
Let $x$ and $I$ be given by hypothesis, and let $y\in \reg(G)$
with $I_y$ the o-minimal open interval witnessing this.
Put $F_x:=I_y\cdot y^{-1}\cdot x$. Then
$F_x$ is an o-minimal, convex open interval around $x$,
and hence $I\cap F_x$ also is a convex o-minimal set.
We have that
\[
F_x\cap I:= \bigcup_{1\leq j\leq t} \left(F_x\cap I\cap C_j\right)
\]
partitions $F_x\cap I$, and by o-minimality is a finite union of
intervals and points. This implies that either $x$ belongs to some
(convex) interval $F_x\cap I\cap C_j$, or $x$ is a limit point of
some $F_x\cap I\cap C_j$ with $x<F_x\cap I\cap C_j$ (by
construction $x$ cannot be a supremum of $F_x$). The assertion that
$x$ the infimum of such a set follows immediately by Fact~\ref{no successor}.
\end{proof}

For a convex set $I$ we define $\cl^-\left(I\right)$ to be the set
consisting of $I$ and, if it exists, the infimum of $I$. Notice that
$\cl^-\left(I\right)\setminus I$ is at most a singleton. We now are
in position to prove the following, which takes us close to the
proof of Theorem~\ref{o-minimal subgroup}.

\begin{lemma}\label{refibered}
Let $G$ be a decomposable ordered group. Then there is a
lexicographic presentation $\langle F_1, \dots , F_n\rangle$ of $G$ such
that each $F_n$-class is an infinite convex o-minimal set.
\end{lemma}

\begin{proof}
Let $\langle E_1,\dots, E_n\rangle$ be a lexicographic presentation of $G$ such
that each $E_1$ class of dimension $n$ are minimal, as guaranteed by
Corollary~\ref{cor:minimal}.

We now restrict the presentation $\langle E_1,\dots, E_n\rangle$ to $\reg(G)$; as we do
not use the original presentation on $G$ in what follows, we continue to use
the same notation.

\begin{claim}
There is a presentation $\langle E'_1, \dots, E_n'\rangle$ of $\reg(G)$ refining
$\langle E_1, \dots , E_n\rangle$ such that the $E'_n$-classes are precisely the
convex (connected) components of the original $E_n$-classes.
\end{claim}

\begin{proof}
Define a new equivalence relation $E_{n+1}$ on the $E_n$ classes by
\[
aE_{n+1}b \hbox{\ if and only if\ }\forall x\, \left(a\leq x\leq b\Rightarrow xE_n a\right).
\]

By the definition of $\reg(G)$ and o-minimality, each $E_n$-class
contains only finitely many $E_{n+1}$ classes, all of which are
infinite convex o-minimal intervals. By uniform finiteness, there
is an absolute bound $k$ for the number $E_{n+1}$-classes in a
single $E_n$ class. By separating each of the $E_{n+1}$-classes
we can partition $\reg(G)$ into the sets $C_1, \dots, C_k$ such that
for all $1\leq i\leq k$ we have
\begin{itemize}
\item[a.] $aE_{n+1} b \Rightarrow a\in C_i\Leftrightarrow b\in C_i$

\item[b.] $a E_n b\wedge a,b\in C_i \Rightarrow aE_{n+1} b$.
\end{itemize}

We then satisfy the claim by defining $E_j'$ for $j<n$  via
\[
aE_j' b \hbox{\ if and only if\ } aE_j b \wedge \bigwedge_{i\leq k} a\in C_i\Leftrightarrow b\in C_i
\]
and setting
$$aE_n'b\Leftrightarrow aE_{n+1}b.$$
\end{proof}

Let $\langle F_1, \dots F_n\rangle$ be the presentation of $G$ given by
\[
a F_i b \;\Leftrightarrow\;\exists x,y\in \reg(G) \left( xE_i' y \wedge
a\in \cl^-\left(\left[x\right]_{E_n'}\right) \wedge b\in
\cl^-\left(\left[y\right]_{E_n'}\right)\right)
\]
for each $i\leq n$.
By Lemma~\ref{limit} we see that $\langle F_1, \dots F_n\rangle$ is indeed a
presentation of $G$ and it follows easily that
this presentation of $G$ has the desired properties.
\end{proof}

We can now prove the main result of this section.

\begin{theorem}\label{o-minimal subgroup}
Let $(G,<,\cdot)$ be an ordered group that is decomposable. Then
there is a convex normal o-minimal subgroup $H\unlhd G$.
\end{theorem}

\begin{proof}
By Lemma \ref{refibered} we may assume that we are given  a
lexicographic decomposition $\langle E_1, \dots E_n\rangle$ of $G$ such that each
$E_n$-class is an infinite o-minimal convex subset of $G$.
Throughout this proof, we refer to these classes as the
\emph{o-minimal fibers of $G$}.

Let $F_0$ be the $E_n$-class containing the identity element $e$ and denote by $F_0^+$ the
positive elements in $F_0$. We refer to a convex subset of $G$ that has $e$ as a left endpoint as a \emph{ray of $G$}.
For a ray $X$ put $\sq(X):=\{x^2| x \in X\}$.

\begin{claim}\label{dimSq}
Let $X:=I_1\cup \dots \cup I_k$ be a ray such that each $I_j$ for $j\leq k$ is a convex
o-minimal set, and let $\overline{\sq(X)}$ be the convex closure of
$\sq(X)$. Then $\overline{\sq(X)}\setminus \sq(X)$
is finite. In particular, $\overline{\sq(X)}$ has dimension~1.
\end{claim}

\begin{proof}
Let $a\in \overline{\sq(X)}\setminus \sq(X)$, and let $\sqrt{a}:=\{y\in X \mid y^2<a\}$.
It is easy to show that $\sqrt{a}$ has no supremum in $X$. Indeed, were $b\in X$ the
supremum of $\sqrt{a}$, then $b^2>a$ and thus for some $w<b$ we have $b\cdot x =a$ Then, for every
$u\in \sqrt a$ with $u>w$ there is some $v<b$ with $u\cdot v=a$. Then the square of the greater of $u$ and $v$
is greater than $a$, which is impossible. Hence by o-minimality and convexity $\sqrt{a}$ must be equal to
$I_1\cup \dots \cup I_j$ for some $j\leq k$, which implies that there are at most $k-1$
elements of $\overline{\sq(X)}\setminus \sq(X)$.
\end{proof}

Let $F_1$ be the set of positive elements in fibers
containing a point in $\overline{\sq(F_0)}$, which by Claim~\ref{dimSq}
has dimension~1. By convexity it must be equal to
a disjoint union of finitely many o-minimal fibers, only the first of which
may not be a complete $E_n$-class.

For $n\geq 1$ we recursively define

\[
F_{n+1}:= \left\{x>0 \mid \exists y\  \left(y\in \overline{\sq(F_n)}\wedge xE_n y\right)\right\}.
\]

An easy induction shows for all $n$ that $F_n$ is a finite union of o-minimal
fibers (together with and the positive elements in the fiber of the identity $e$). But
since $G/E_n$ is a decomposable ordered structure, it cannot have
infinite discrete chains so for some $n$ we have $F_n=F_{n+1}$.

With $F_n=F_{n+1}$ we have $\sq(F_n)\subseteq F_n$, and by convexity it follows that
$F_n^+\cdot F_n^+\subset F_n^+$. Closing $F_n^+$ under inverses we obtain a
one-dimensional convex subgroup $N$ of $G$.
Arguing as in the proof of Claim~\ref{dimSq} above,
we see that $N$ is a convex o-minimal subgroup of $G$.

It remains to prove that $N$ is normal. For this, observe that $aNa^{-1}$ is a convex
subgroup of $G$ of dimension 1 and thus $N_a:=N\cap aNa^{-1}$ is a convex
subgroup of $N$. By o-minimality $N_a=N$, whence $N$ is normal.
\end{proof}

The following is an easy and important consequence of Theorem~\ref{o-minimal subgroup}.

\begin{corollary}\label{exactness}
Let $G$ be a decomposable ordered group. Then there is a normal o-minimal
ordered subgroup $N$ of $G$ and a decomposable ordered group $H$ such that
\[
1\rightarrow N\rightarrow G\rightarrow H\rightarrow 1
\]
is an exact sequence of ordered groups, where each of the homomorphisms is a
homomorphism of ordered groups.
\end{corollary}

\begin{proof}
Theorem \ref{o-minimal subgroup} provides the existence of the normal
o-minimal subgroup $N$, and by convexity $G/N$ is an ordered group.
If the ground theory eliminates imaginaries, the fact that both groups are
decomposable follows from Theorem 4.1 in \cite{OnSt}. In our case, the result
follows by the proof of Theorem \ref{o-minimal subgroup}, in particularly from the
re-fibering given by Lemma \ref{refibered}.
\end{proof}

And from this, an easy induction shows the following.
\begin{theorem}\label{decomposable}
Every ordered decomposable group $G$ is supersolvable and
homeomorphic to a finite product of o-minimal structures. Even more,
any ordered decomposable group $(G, \odot, <_G)$ has a supersolvable
decomposition
$$H_0\unlhd H_1\unlhd H_2\unlhd\dots \unlhd H_n=G$$
such that $H_{i+1}/H_i$ is an o-minimal group, $H_i$ is a normal subgroup of $G$, and the order $<_G$ is
given by the lexicographic order induced by the order on the
quotients $H_{i+1}/H_i$.
\end{theorem}

\section{Cohomology of groups definable in o-minimal theories}\label{consequences in o-minimal}\label{section3}

In this section we
survey some of the definable
group cohomology that we shall need; this material was introduced in \cite{Ed} and further developed by the first
author in her M.Sc. thesis written under the supervision of the second author.
Throughout this section when we state that a group is interpretable or definable, we mean, if not stated  explicitly otherwise, that it is interpretable or definable in an o-minimal structure.

A reader familiar with group extensions and the cohomology
of groups may recall that for a group $G$ and normal subgroup $H\unlhd G$, the existence of a section from $G/H$ to $G$ makes a big difference in whether or not $G$ is an extension of $G/H$ by $H$. This hypothesis must be included in the definition of a group
extension in order for there to be an isomorphism between the
respective cohomology groups and the group extensions. With this in mind, we
begin with the following theorem of
Edmundo (see Corollary 3.11 in \cite{Ed}).

\begin{fact}\label{edmundo}
Let $U$ be a definable group in an o-minimal structure and $N\unlhd U$ a
definable normal subgroup. Then there is a definable group $H$,
a definable exact sequence
\[
0\rightarrow N\rightarrow U\rightarrow H\rightarrow 1
\]
and a definable section
\[
s: H\rightarrow U.
\]
\end{fact}

We now make a series of definitions.

\begin{definition}\label{3.2}
Let $H$ be a definable group in an o-minimal structure. A \emph{definable} $H$-{\em module}
$\left(N,\gamma\right)$ is a $H$-module such that $N$ is a definable abelian group and the action map  $\gamma:H\times N\rightarrow N$ given by $\gamma\left(x,a\right)\coloneqq \gamma\left(x\right)\left(a\right)$ is definable. As usual $\gamma$ is usually called an {\em action} of $H$ on $N$ and $\gamma$ is {\it trivial\/} if  $\gamma\left(g\right)\left(a\right)=a$ for all $g\in H$ and $a\in N$.
\end{definition}

Observe that an action $\gamma$ as above induces homomorphism
$\gamma:H\rightarrow \aut\left(N\right)$ from $H$ to the {\em group of all
definable automorphisms of $N$\/}.

\begin{definition}
Two actions $\gamma_{1}$ and $\gamma_{2}$ of $H$ on $N$  are {\em definably equivalent\/} if there is a definable group automorphism $\psi$ of $H$ such that $\gamma_{1}=\gamma_{2}\circ\psi$.
\end{definition}

\begin{definition}
Let $H$ be a definable group and let $N$ be a definable $H$-module; equivalently, let
$\gamma:H\rightarrow \aut(N)$ be a definable action of $H$ on $N$.
We say that
$U$ is \emph{an extension of $H$ by $N$ via $\gamma$\/} if there
is an exact sequence
\[
1\rightarrow N\xrightarrow{i} U\xrightarrow{\pi} H\rightarrow 1
\]
such that for all $h\in H$, $g\in U$ with $\pi (g)=h$,
and $n\in N$ we have
\[
\gamma(h)(n):=gng^{-1}.
\]
We often use the notation $(U,i,\pi)$ to denote an extension of $H$ by $N$ as above, and also
write $\gamma(h)(n)$ as $n^{\gamma(h)}$.
\end{definition}

\subsection{The cochain complex for a $G$-module and cohomology}

We here recall some fundamental properties of o-minimal
group cohomology as defined in \cite{Ed}. Let $H$ be an o-minimal definable group, and let
$\left(N,\gamma\right)$ be an $H$-module with action $\gamma$. To
distinguish the operations on $H$ and $N$, we write the operation on $H$
multiplicatively while the operation on $N$ is written additively.

The next definitions in the o-minimal case can be found in \cite{Ed}.

\begin{definition}
For $n\in\mathbb{N}\setminus \{0\}$, we denote by
$C^{n}\left(H,N,\gamma\right)$ the (abelian) group consisting of the set of all
functions from $H^{n}$ to $N$, and for $n=0$ we put
$C^{0}\left(H,N\right)=N$. The elements of $C^{n}\left(H,N,\gamma\right)$ are called $n$-\emph{cochains}. If the action
$\gamma$ is clear from the
context we sometimes write $C^{n}\left(H,N\right)$ instead of $C^{n}\left(H,N,\gamma\right)$.
\end{definition}

\begin{definition}\label{D:complejo}\label{H}
The \emph{coboundary function} $\delta:C^{n}\left(H,N,\gamma\right)\rightarrow C^{n+1}\left(H,N,\gamma\right)$ is defined by
\begin{align*}
\delta\left(f\right)\left(g_{1},\ldots,g_{n+1}\right) = &\; \gamma\left(g_{1}\right)\left(f\left(g_{2},\ldots,g_{n+1}\right)\right)\\ &+
\sum_{i=1}^{n}\left(-1\right)^{i}f\left(g_{1},\ldots,g_{i}g_{i+1},\ldots,g_{n+1}\right)\\
& +\left(-1\right)^{n+1}f\left(g_{1},\ldots,g_{n}\right).
\end{align*}
\end{definition}

\noindent It is clear that ${\delta\left(f\right)}\in {{C^{n+1}}\left(H,N,\gamma\right)}$ and
as is standard $\delta\circ\delta=0$. Thus $\left(C,\delta\right)$ is a {\em definable cochain complex\/}
and we have

\begin{definition}
For $n\geq0$ the \emph{cohomology groups of the complex} $\left(C,\delta\right)$ are given by
$$H^{n}\left(H,N\right)=H^{n}\left(H,N,\gamma\right)
= \frac{Z^{n}\left(H,N,\gamma\right)}{B^{n}\left(H,N,\gamma\right)},$$
where
\[
Z^{n}\left(H,N,\gamma\right)={\ker\delta:C^{n}\left(H,N,\gamma\right)\rightarrow C^{n+1}\left(H,N,\gamma\right)},\]
$B^{0}{\left(H,N,\gamma\right)}= {\left\{ 0\right\}}$ and for $n>0$
\[B^{n}\left(H,N,\gamma\right)={\textrm{im}\,{\delta:C^{n-1}\left(H,N,\gamma\right)\rightarrow C^{n}\left(H,N,\gamma\right)}}.
\]
The elements of $Z^{n}\left(H,N,\gamma\right)$ are the \emph{cocycles} and the elements of $B^{n}\left(H,N,\gamma\right)$ \emph{coboundaries} of the cochain complex.
\end{definition}

By standard methods in group cohomology---employed by
Edmundo in \cite{Ed} for o-minimal group cohomology---it can be shown
that if the action $\gamma\colon H\to \aut(N)$ is fixed, then the group of
classes of group extensions of $H$ by $N$ is isomorphic to group of classes of 2-cocycles in
$H^2(H,N, \gamma)$. Even
more, given such a 2-cocycle one can recover the group operation on the extension.
%
In the o-minimal context we have the important result from  \cite{Ed}.

\begin{fact}\label{ExtensionsAndCohomology}
Let $H,N$ be groups definable in an o-minimal structure, and let
$\gamma$ be a definable group action of $H$ on $N$. Then the group of equivalence classes
of group extensions of $H$ by $N$ via $\gamma$ (modulo isomorphism)
is isomorphic to the the group of equivalence classes of definable  2-cocycles in $H^2(H,N,
\gamma)$. In fact, given a definable 2-cocycle $c$ one recovers the
group operation by
\[(a,g)\cdot(b,h):= (a+b^{\gamma(g)}+c(g,h), gh)
\]
for all $(a,g), (b,h)\in N\times H$.
\end{fact}

We now make an easy observation that allows us to move back and forth
between the definable group context and the definable ordered group
context.

\begin{proposition}\label{order preserving action}
Let $G$ be an ordered group and $N$ be a convex normal subgroup,
both definable in an o-minimal structure. Then $G$ is an extension
of $G/N$ by $N$ with an action given by a morphism from $G/N$ into
the \emph{order preserving} automorphisms of $H$.
\end{proposition}

\begin{proof}
For any  $g\in G$ and any $n_1<n_2\in N$ we have, by the
definition of an ordered group, that $gn_1g^{-1}<gn_2g^{-1}$. Since
the action of $h\in G/H$ on $N$ is given by conjugation of any
representative $g$ of the class of $h$, the proposition follows.
\end{proof}

The following theorem is the main result of this section.
Together with Proposition~\ref{order preserving action} it allows us to use cohomology
of groups to characterize ordered groups definable in o-minimal
structures.

\begin{theorem}\label{2-Cohomology}
Let $(H, \cdot)$ and $(N,+)$ be two definable ordered groups and let
$\gamma:H\rightarrow \aut(N)$ be a definable group action of $H$
on $N$ such that  for all $h\in H$ we have that
$\gamma(h)$ is an order preserving definable
automorphism of $N$. Then there is a one-to-one
correspondence between the 2-cocycles in $H^2(H,N,\gamma)$ and the
group extensions of $H$ by $N$ which, when endowed with the
lexicographic order $<_l$ given first by $H$ and then by $N$,
become ordered groups.
\end{theorem}

\begin{proof}
The right-to-left direction of the correspondence is given by the classic result for groups and Proposition~\ref{order preserving action}.

For the other direction, let $c(x,y)\in Z^2(H,N, \gamma)$.
We define a group structure on $N\times H$ by
\[(a,g)\odot(b,h):= (a+b^{\gamma(g)}+c(g,h), gh).
\]
We need only show that if $(b,h)<_l(d,j)$ then $(a,g)\odot
(b,h)<_l(a,g)\odot (d,j)$. Since $H$
is an ordered group, if $h<_H j$ then $gh<_H gj$ then by definition of
the lexicographic order $<_l$ we are done.

So we may assume that $h=j$ and $b<_N d$. Since $gh=gj$ we must
prove that
\[
a+b^{\gamma(g)}+c(g,h)<a+d^{\gamma(g)}+c(g,j).
\]

This is immediate: $c(g,h)=c(g,j)$ and by hypothesis
$\gamma(g)$ for $g\in H$ is an order preserving automorphism of $N$, whence
$b<_N d$ implies $b^{\gamma(g)}<_N d^{\gamma(g)}$.
\end{proof}

\section{Restricting the class of structures: The Pfaffian closure of the
real field}\label{subsection3.2}

In the (pure) real field, the additive group $(\Rr,+)$ and
multiplicative subgroup $(\Rr^+, \cdot)$ are not definably isomorphic. This
implies, for example, that the four groups $(\Rr,+)\times (\Rr,+)$,
$(\Rr^+, \cdot)\times (\Rr,+)$, $(\Rr,+)\times (\Rr^+, \cdot)$, and
$(\Rr^+, \cdot)\times (\Rr^+, \cdot)$, ordered lexicographically, are not
definably isomorphic. Of course, in the real exponential field
$(\mathbb R,+, \cdot, e^x)$ these four groups become definably
isomorphic. While our overarching goal is to characterize ordered groups definable in an
o-minimal structure, we prefer not to
focus on structures in which the existence of such
isomorphisms is an issue.\footnote{The interested reader is invited to peruse the excellent survey
\cite{MiSt} for more on this topic.}

We therefore concentrate in what follows on o-minimal
structures in which this is not a concern.
The issue of whether a definable isomorphism between the additive and multiplicative
groups exists in an o-minimal expansion of the real field is not specific to these groups.
The following result appears in \cite{PSS}, but follows from results in \cite{Ed} and \cite{Sp}.

\begin{fact}
Let $\mathcal M$ be an o-minimal expansion of
the field of real numbers in which there are two nonisomorphic definable o-minimal
groups $H$ and $G$. Then there is an ordered group isomorphism $\phi\colon H\to G$
such that the structure $(M, \phi)$ is o-minimal. In fact, if $G$ is the additive group of the real field, then $\phi$ lies in the Pfaffian closure of the real field.
\end{fact}

Hence, in the Pfaffian closure of the real field all o-minimal
definable groups are isomorphic. In view of the foregoing discussion, we
find it natural to expand the language so all such groups are
definably isomorphic and therefore in what follows we assume that
\smallskip
\begin{equation*}\label{isom assumption}
(*)\qquad\parbox{4.5in}{\emph{we work in an o-minimal field $\Rc$ such that all
o-minimal definable subgroups of $\Rc$ are definably isomorphic.\/}}
\end{equation*}
In particular, notice that $\Rc$ is an exponential real closed
field. This assumption yields an immediate sharpening of Theorem
\ref{decomposable}.

\begin{corollary}\label{solvable groups}
Let $G$ be an ordered group definable in $\Rc$ as above with universe $\Rr$.
Then $G$ is supersolvable and homeomorphic to $\Rr^n$. Specifically, $G$ admits
a supersolvable chain $H_0\unlhd H_1\unlhd H_2\unlhd\dots \unlhd
H_n=G$ such that $H_{i+1}/H_i$ is $\mathcal R$-isomorphic to
$(\mathbb R,+)$ and the order $<_G$ is given by the lexicographic order
induced by the order on the quotients $H_{i+1}/H_i$.
\end{corollary}


\section{Dimension two groups} 
\label{dimension2}

We now apply Theorem~\ref{2-Cohomology} to analyze
groups of dimension two in $\mathcal R$, an o-minimal expansion of a real closed
field satisfying~$(*)$, that is, in which all o-minimal groups are isomorphic. All objects in this section are assumed to be
definable in $\mathcal R$. Even though all
the results in this section hold in this context, for concreteness we shall take $\mathcal R$ to be
an o-minimal expansion of the field of real numbers $\mathbb R$ in which all o-minimal groups are isomorphic. We also shall abuse notation by referring to
$\mathbb{R}$ as the additive group and by $\mathbb R^n$ as the cartesian
product of $n$ copies of the additive group $\mathbb R$. Note from~$(*)$ that if $N$ is an o-minimal
group definable in $\mathcal R$ then $\aut(N)$, the group of \emph{definable}
automorphisms of $N$, is isomorphic to $(R\setminus 0, \cdot)$. In
particular, the group of {\em order preserving\/} automorphisms is
(isomorphic to) the connected component $(R^{>0},\cdot)$, which is
definably isomorphic to $(R,+)$.

We also shall have to adopt an additional hypothesis for the results we prove here.
Let $\left(\mathbb{R},+\right)$ a definable
$\left(\mathbb{R},+\right)$-module with action $\gamma$ and
a definable extension
\[
1\rightarrow \mathbb{R}\rightarrow G\rightarrow
\mathbb{R}\rightarrow 1
\]
be a definable extension. The issue of whether the extension splits is delicate;
in general we do not know if it splits if $G$ is abelian, see \cite{PS05} (in this situation
it is not difficult to show that the nonabelian case splits). Thus
we need to add another condition on our structure $\mathcal{R}$:
\begin{equation*}\label{isom assumption2}
(**)\qquad\parbox{4.5in}{\emph{In
$\mathcal{R}$ every abelian definable extension of $\mathbb{R}$ by
$\mathbb{R}$ splits.\/}}
\end{equation*}

Under this assumption, combining
the known nonabelian case with Fact~\ref{ExtensionsAndCohomology}), we find that
$$H^{2}\left(\mathbb{R},\mathbb{R},\gamma\right)=\left\{ 0\right\}$$
for all actions $\gamma$ of $\mathbb{R}$ on $\mathbb{R}$.

Our goal in this section is Theorem~\ref{T:1.3} in which we
completely describe all extensions of $\mathbb R$ by $\mathbb R$,
and moreover all decomposable ordered groups of dimension 2
definable in $\mathcal R$ under the assumptions~$(*)$ and~$(**)$.
The key ingredient in the proof of this
theorem is Theorem \ref{T:1.2}, which asserts that for any two definable nontrivial actions $\gamma$ and
$\gamma^{\prime}$ of $\mathbb{R}$ on $\mathbb{R}$, the extensions
$\mathbb{R}\rtimes_{\gamma}\mathbb{R}$ and
$\mathbb{R}\rtimes_{\gamma^{\prime}}\mathbb{R}$ are definably
isomorphic groups.

We require the following two facts about definable group
actions that can be found in \cite{Ba}.

\begin{fact}\label{P:accionesgrales}
Let $\left(\mathbb{R},+\right)$ be a definable
$\mathbb{R}^{n}$-module with nontrivial action $\gamma$. Then
\[\gamma\left(x_{1},\ldots ,x_{n}\right)=e^{c_{1}x_{1}+\ldots+c_{n}x_{n}},\]
where $c_{i}\neq0$ for some $i\in\left\{ 1,\ldots,n\right\}$.
\end{fact}

\begin{fact}\label{P:accionesequi}
Let  $\gamma^{\prime}\left(\left(x_{1},\ldots,x_{n}\right)\right)=e^{c_{1}x_{1}+\ldots+c_{n}x_{n}}$
and
$\gamma\left(\left(x_{1},\ldots,x_{n}\right)\right)=e^{x_{1}}$
be nontrivial actions for which $\left(\left(\mathbb{R},+\right),\gamma\right)$ and
$\left(\left(\mathbb{R},+\right),\gamma^{\prime}\right)$ are
definable  $\mathbb{R}^{n}$-modules. Then $\gamma$ and
$\gamma^{\prime}$ are definably equivalent.
\end{fact}

We now have

\begin{theorem}\label{T:1.2}
Let $\left(\left(\mathbb{R},+\right),\gamma\right)$ and
$\left(\left(\mathbb{R},+\right),\gamma^{\prime}\right)$ be
definable $\mathbb{R}$-modules with nontrivial actions $\gamma$ and
$\gamma^{\prime}$. Then $\mathbb{R}\rtimes_{\gamma}\mathbb{R}$ and
$\mathbb{R}\rtimes_{\gamma^{\prime}}\mathbb{R}$ are definably
isomorphic.
\end{theorem}
\begin{proof}
By Fact~\ref{P:accionesgrales},
we may assume that $\gamma\left(y\right)=e^{cy}$ and $\gamma^{\prime}\left(y\right)=e^{dy}$,
for non-zero real constants $c$ and $d$.
We now assert that the bijection
\begin{eqnarray*}
\mathbb{R}\rtimes_{\gamma}\mathbb{R} & \rightarrow & \mathbb{R}\rtimes_{\gamma^{\prime}}\mathbb{R}\\
\left(x,y\right) & \mapsto & \left(x,\frac{c}{d}y\right)\end{eqnarray*}
is a group isomorphism. Indeed

\[\left(x_{1},y_{1}\right)\otimes\left(x_{2},y_{2}\right)  =  \left(x_{1}+e^{cy_{1}}x_{2},y_{1}+y_{2}\right)\mapsto\left(x_{1}+e^{cy_{1}}x_{2},\frac{c}{d}\left(y_{1}+y_{2}\right)\right)\]
and
\[\left(x_{1},\frac{c}{d}y_{1}\right)\left(x_{2},\frac{c}{d}y_{2}\right)  =  \left(x_{1}+e^{cy_{1}}x_{2},\frac{c}{d}\left(y_{1}+y_{2}\right)\right).\]

\end{proof}


Notice in the proof above that if $\gamma$ and $\gamma^\prime$ are
order preserving, then by Proposition~\ref{order preserving action}
and Theorem~\ref{2-Cohomology} the groups are ordered groups
under the lexicographic order, and it is easy to verify that the map
above preserves this ordering if and only if $c/d>0$. We here restrict our attention to
understanding the group isomorphisms between the ordered groups that
arise. In Section~\ref{sectionOrdered} we work to understand
these ordered groups modulo \emph{ordered} isomorphism.

\begin{remark}
In view of Theorem~\ref{T:1.2}, when we write the
semidirect product $\mathbb{R}\rtimes\mathbb{R}$, we
take the action to be $\gamma\left(x\right)=e^{x}$.
\end{remark}

\begin{theorem}\label{T:1.3}
Let $\left(\mathbb{R},+\right)$ be a definable $\mathbb{R}$-module.
Then, modulo definable isomorphisms, the only possible definable
extensions of $\mathbb{R}$ by $\mathbb{R}$ are
$\left(\mathbb{R}\times\mathbb{R},+\right)$ and
$\mathbb{R}\rtimes\mathbb{R}$. Furthermore, the only possible
decomposable groups of dimension~2 definable in $\mathcal R$ are,
modulo definable group isomorphisms,
$\left(\mathbb{R}\times\mathbb{R},+\right)$ and
$\mathbb{R}\rtimes\mathbb{R}$, both ordered lexicographically.
\end{theorem}

\begin{proof}
The first assertion is a direct consequence of our assumptions,
Theorem \ref{T:1.2}, and the fact that $\mathbb{R}\rtimes\mathbb{R}$
is not isomorphic to $\left(\mathbb{R}\times\mathbb{R},+\right)$.
For the second assertion, we first note that Theorem \ref{o-minimal subgroup} implies the existence of a unique convex one dimensional o-minimal subgroup. Then
applying Proposition~\ref{order preserving
action}, and Theorems~\ref{2-Cohomology} and~\ref{T:1.2}, it follows that
any two definable ordered groups with underlying domain
$\left(\mathbb{R}\times\mathbb{R},+\right)$,
respectively, $\mathbb{R}\rtimes\mathbb{R}$, are isomorphic as ordered groups.
\end{proof}


\section{Spectral sequences}\label{spectral}

Our goal in this and the ensuing section is to classify all dimension~3 ordered groups definable
an o-minimal expansion of a real closed field $\Rc$, under the ongoing assumptions~$(*)$
and~$(**)$.  Generally speaking, we shall work to understand all such groups inductively, by analyzing
them as extensions of their convex normal o-minimal subgroup $N$ and
the (2-dimensional) quotient $H:=G/N$.

Before we can analyze the cases that arise, we require some facts about the cohomology groups of
definable modules in $\mathcal R$. Some of this work relies on the study of
spectral sequences that the first author developed in her M.Sc.thesis; all the background needed is developed in
 \cite{BarrigaThesis} and can be
found in \cite{Ba}. The notation is standard and may be found in, e.g., \cite{We}.

As in Section~\ref{dimension2}, all objects are assumed to be definable in
$\mathcal R$, and although our results hold under in general under the hypotheses~$(*)$
and~$(**)$, for concreteness we take the underlying domain of $\mathcal R$ to be
$\mathbb R$ (such as if we work in $\mathbb R_{{\it Pfaff}}$). In fact we abuse notation
in what follows by referring to $\mathbb{R}$ as the additive group of the field and we also write
$\left(\mathbb{R}^{2},+\right)$ for
$\left(\mathbb{R}\times\mathbb{R},+\right)$.

The following, which
appears as Corollary~3.8 in \cite{Ba}, plays a crucial role.

\begin{fact}\label{C:HyS}
Let $\left(M,\gamma\right)$ be a definable $G$-module and $K$ a definable normal subgroup of $G$. For the definable  cochain
complex
$\left(C=\underset{n\geq0}{\oplus}C^{n}\left(G,M,\gamma\right),\delta\right)$
there are filtrations of the groups
$H^{1}=H^{1}\left(G,M,\gamma\right)$ and
$H^{2}=H^{2}\left(G,M,\gamma\right)$ of the form
\[
\left\{ 0\right\} \leq F^{1}H^{1}\trianglelefteq H^{1}
\]and \[
\left\{ 0\right\} \trianglelefteq F^{2}H^{2}\trianglelefteq F^{1}H^{2}\trianglelefteq H^{2}
\] such that
\[
F^{1}H^{1}=H^{1}\left(\nicefrac{G}{K},M^{K}\right) \text{,}\quad
{H^{1}}/{F^{1}H^{1}}\leq
\left(H^{1}\left(K,M\right)\right)^{\nicefrac{G}{K}},
\]
and
\begin{eqnarray*}
F^{2}H^{2}&=&\frac{H^{2}\left(\nicefrac{G}{K},M^{K}\right)}{{\rm im}\: d:\left(H^{1}\left(K,M\right)\right)^{\nicefrac{G}{K}}\rightarrow H^{2}\left(\nicefrac{G}{K},M^{K}\right)}, \\
\smallskip
{F^{1}H^{2}}/{F^{2}H^{2}}
&\leq & H^{1}\left(\nicefrac{G}{K},H^{1}\left(K,M\right)\right),\\
{H^{2}}/{F^{1}H^{2}}
&=&\left(H^{2}\left(K,M\right)\right)^{\nicefrac{G}{K}}.
\end{eqnarray*}
\end{fact}
\smallskip

If $G$ is a group acting on the additive group
$(\mathbb R, +)$, then $C^n(G, \mathbb R)$ is a
$\mathbb R$-module by definition, a property that is inherited by the
quotients $H^n(G,\mathbb R)$.
With this observation, the following follows immediately from Fact~\ref{C:HyS}.

\begin{corollary}\label{all you need}
Let $G$ be a group with an action $\gamma$ on the additive group
$(\mathbb R, +)$, the additive group of $\mathcal R$, and let $K$ be
a definable normal subgroup of $G$. Then the following hold:

\begin{enumerate}
\item
If $H^{2}\left(\nicefrac{G}{K},\mathbb R^{K}\right)$,
$H^{1}\left(\nicefrac{G}{K},H^{1}\left(K,\mathbb R\right)\right)$,
and $\left(H^{2}\left(K,\mathbb R\right)\right)^{\nicefrac{G}{K}}$
are all trivial groups, then $H^{2}\left(G,\mathbb R,\gamma\right)$
is trivial.

\item
If one of $H^{2}\left(\nicefrac{G}{K},\mathbb R^{K}\right)$,
$H^{1}\left(\nicefrac{G}{K},H^{1}\left(K,\mathbb R\right)\right)$,
and $\left(H^{2}\left(K,\mathbb R\right)\right)^{\nicefrac{G}{K}}$
is isomorphic to $\mathbb R$ and the other two are trivial groups,
then $H^{2}\left(G,\mathbb R,\gamma\right)$ is either the identity
or isomorphic to $\mathbb R$.
\end{enumerate}
\end{corollary}

\begin{proof}
The first assertion is clear. For the second,
since a nontrivial (vector)-subspace of a field must be the field itself, the only sub-$\mathbb R$-module of $\mathbb R$ other than the identity, is $\mathbb R$ itself.
\end{proof}

\section{Ordered decomposable groups of dimension 3}\label{dimension3}


Now we classify all dimension~3 ordered groups definable in an
o-minimal expansion of a real closed field satisfying the ongoing
assumptions~$(*)$ and~$(**)$. This analysis ultimately
breaks up into cases as follows. Suppose that $G$ is an ordered
group of dimension~3 in $\Rc$. By Theorem~\ref{o-minimal subgroup},
there is a convex, normal o-minimal subgroup $N$ of $G$ such that
$H:=G/N$ under the inherited ordering is one of the two definable
ordered subgroups of dimension~2 (up to isomorphism) determined in
Theorem~\ref{T:1.3}. Moreover the natural conjugation of $N$ by $H$
(bear in mind that $N$ is abelian) is a definable order-preserving
action making $N$ into a definable $H$-module. Our classification
now divides into cases depending on $H$ and on whether or not the
action is trivial, and each possibility will be treated in turn.

In view of the facts presented in Section \ref{spectral}, it is important to
understand the cohomology groups of smaller dimensional group
extensions.

\begin{lemma}\label{smaller dimensional groups}
Each of the following holds.

\begin{enumerate}
\item Let $\left(\mathbb{R},\gamma\right)$ be a definable $G$-module.  If $\gamma$ is the trivial action, then\\ $H^{0}\left(G,\mathbb{R},\gamma\right)=\mathbb{R}$. If $\gamma$ is a nontrivial action, then $H^{0}\left(G,\mathbb{R},\gamma\right)\allowbreak =\left\{ 0\right\} $.
\item Let $\left(\mathbb{R},+\right)$ be the definable $\left(\mathbb{R}^{n},+\right)$-module with the trivial action. Then
$H^{1}\left(\mathbb{R}^{n},\mathbb{R}\right)\simeq
\left(\mathbb{R}^{n},+\right)$.
\item Let $\left(\mathbb{R},+\right)$ be a definable $\mathbb{R}^{n}$-module with a nontrivial action. Then \\ $H^{1}\left(\mathbb{R}^{n},\mathbb{R}\right) =\left\{ 0\right\}$.
\item Let $\left(\mathbb{R},+\right)$ be a definable
$\left(\mathbb{R},+\right)$-module. Then $H^{2}\left(\mathbb{R},\mathbb{R}\right)=\{0\}$.

\end{enumerate}

\end{lemma}

\begin{proof}
The first three items are proved in \cite{Ba}. They follow easily
from the facts that
$H^{0}\left(G,M,\gamma\right)=M^{G}=\left\{ m\in M:\left(\forall
g\in G\right)\left(\gamma\left(g\right)x=x\right)\right\} $ and that
$H^{1}\left(G,M,\gamma\right)$ is the set of crossed homomorphisms
from $G$ to $M$ modulo principal homomorphisms, as in the general
group homomorphism setting. 

The last item follows immediately from Theorems \ref{2-Cohomology}
(or Fact \ref{ExtensionsAndCohomology}) and our assumptions.
\end{proof}

We also shall avail ourselves of the following facts (\cite[Propositions 2.3, 2.4]{Ba}), both easy consequences of the o-minimality of $\mathcal R$.

\begin{fact}\label{P:homen1iguales}
Let $G$ be a definable group in $\mathcal R$ whose operation is denoted multiplicatively. Let $f$, $g$ be two definable homomorphisms from $\left(R,+\right)$ into $G$. If $f\left(1\right)=g\left(1\right)$ then $f=g$.
\end{fact}

\begin{fact}\label{P:homdefcont}
Let $\left(G,\oplus\right)$ and $\left(G^{\prime},\odot\right)$ definable groups in $\mathcal R$ with
$G\subseteq R^m$, $G^{\prime}\subseteq R^{n}$, and $G$ infinite. If $\oplus$ and $\odot$ are continuous and $f:\left(G,\oplus\right)\rightarrow\left(G^{\prime},\odot\right)$ is a definable homomorphism, then $f$ is continuous.
\end{fact}



\medskip

We will now start with a case by case analysis of the different
possible 3-dimensional ordered groups.

\subsection{$H=\mathbb{R}\times\mathbb{R}$ and the action of $H$
on $\mathbb{R}$ is trivial}\label{HeisSect}\hspace*{\fill} \\

We wish to prove that
$H^{2}\left(\mathbb{R}^{2},\mathbb{R}\right)\simeq\left(\mathbb{R},+\right)$.
As a first step we apply
Lemma \ref{smaller dimensional groups} and Corollary \ref{all you need} to bound this cohomology group.

\begin{lemma}\label{T:2.2}
Let $\left(\mathbb{R},+\right)$ be a definable $\left(\mathbb{R}^{2},+\right)$-module  with the trivial action. Then
$H^{2}\left(\mathbb{R}^{2},\mathbb{R}\right)\leq\left(\mathbb{R},+\right)$.
\end{lemma}
\begin{proof}
Let $K=\left\{ \left(0,x\right)\vert x\in\mathbb{R}\right\}
\trianglelefteq\left(\mathbb{R}^{2},+\right)$. It is easy to see
for all $j\geq 0$ that $H^{j}\left(K,\mathbb{R}\right)$ is a
$\nicefrac{\mathbb{R}^{2}}{K}$-module with trivial action.
By Lemma \ref{smaller dimensional groups} we have

\begin{eqnarray*}
H^{2}\left(\nicefrac{\mathbb{R}^{2}}{K},\mathbb{R}^{K}\right)&=& H^{2}\left(\mathbb{R},\mathbb{R}\right)=\left\{ 0\right\} \\
H^{1}\left(\nicefrac{\mathbb{R}^{2}}{K},H^{1}\left(K,\mathbb{R}\right)\right)&=& H^{1}\left(\mathbb{R},\mathbb{R}\right)=\left(\mathbb{R},+\right) \\
H^{0}\left(\nicefrac{\mathbb{R}^{2}}{K},H^{2}\left(K,\mathbb{R}\right)\right)&=& H^{2}\left(\mathbb{R},\mathbb{R}\right)^{\mathbb{R}}=\left\{0\right\}.
\end{eqnarray*}
The lemma now follows by Corollary \ref{all you need}.
\end{proof}

To complete the proof that
$H^{2}\left(\mathbb{R}^{2},\mathbb{R}\right)\simeq\left(\mathbb{R},+\right)$ we produce an extension that is not isomorphic to the trivial one.



\begin{definition}[{The Heisenberg group}]
The Heisenberg group is the group $SUT\left(3,\mathbb{R}\right)$,
the set of all the upper triangular matrices with real entries of the form
\[
\left[\begin{array}{ccc}
1 & x & z\\
0 & 1 & y\\
0 & 0 & 1\end{array}\right]\]
with the group operation given by matrix multiplication.
\end{definition}

\noindent For the calculations that we carry below, we use the following presentation of the Heisenberg group. Let the operation $\cdot$ on $\mathbb{R}^{3}$  be given by
\[
\left(x_{1},y_{1},z_{1}\right)\cdot\left(x_{2},y_{2},z_{2}\right) =
\left(x_{1}+x_{2},y_{1}+y_{2},z_{1}+z_{2}+\frac{1}{2}
\det\left[\begin{array}{cc} x_{1} & y_{1}\\x_{2} &
y_{2}\end{array}\right] \right).
\]
With this operation $\left(\mathbb{R}^{3},\cdot\right)$ is a group that we denote by
$\mathcal{G}_{\textrm{Heis}}$.

Observe that $SUT\left(3,\mathbb{R}\right)\simeq\mathcal{G}_{\textrm{Heis}}$ via the isomorphism
\[
\left[\begin{array}{ccc}
1 & x & z\\
0 & 1 & y\\
0 & 0 &
1\end{array}\right]\mapsto\left(x,y,z-\frac{1}{2}xy\right).\]

The construction of $\mathcal{G}_{\textrm{Heis}}$ can be generalized as follows.

\begin{definition}

For $c\in\mathbb{R}\setminus\left\{ 0\right\} $, let $E_{c}$ be the group on $\mathbb{R}^{2}\times\mathbb{R}$ with product given by
\[
\left(\left(x_{1},y_{1}\right), z_{1}\right)\cdot\left(\left(x_{2},y_{2}\right), z_{2}\right)
= \left(\left(x_{1}+x_{2},y_{1}+y_{2}\right), z_{1}+z_{2}+c\det\left[\begin{array}{cc}x_{1} & y_{1}\\
x_{2} & y_{2}\end{array}\right]\right).
\]

\end{definition}

The next assertion is an easy calculation.

\begin{proposition}\label{Heissenberg}
Let $i\left(z\right)=\left(\left(0,0\right), z\right)$,
$\pi\left(\left(x,y\right), z\right)=\left(x,y\right)$, and
$c\in\mathbb{R}\setminus\left\{ 0\right\} $. Then
$0\rightarrow\mathbb{R}\overset{i}{\rightarrow}
E_{c}\overset{\pi}{\rightarrow}\mathbb{R}^{2}\rightarrow0$ is a definable extension of $\mathbb{R}^{2}$ by $\mathbb{R}$.
\end{proposition}

We thus have

\begin{theorem}\label{T:2.3}
Let $\left(\mathbb{R},+\right)$ be a definable $\left(\mathbb{R}^{2},+\right)$-module with the trivial action. Then
$H^{2}\left(\mathbb{R}^{2},\mathbb{R}\right)\simeq\left(\mathbb{R},+\right)$.
\end{theorem}

\begin{proof}
Recall that $H^{2}\left(\mathbb{R}^{2},\mathbb{R}\right)$
carries the structure of an $\mathbb R\/$-vector space. Indeed, for $[f]$ a nonzero element of
$H^{2}\left(\mathbb{R}^{2},\mathbb{R}\right)$ let $\alpha [f]:= [\alpha f]$, where $\alpha\in\mathbb R$. Combining this fact and Lemma~\ref{T:2.2}
we see that $H^{2}\left(\mathbb{R}^{2},\mathbb{R}\right)$ is either $\{0\}$ or $\mathbb R$. Proposition~\ref{Heissenberg} above provides
a nontrivial extension, hence $H^{2}\left(\mathbb{R}^{2},\mathbb{R}\right)\simeq\mathbb R$.
\end{proof}

One can in fact show more: the family of nontrivial extensions is given
precisely by the family of groups $E_c$.

\begin{proposition}\label{P:EcEd} Let $c,d\in\mathbb{R}\setminus\left\{ 0\right\} $. If $c\neq d$ then the extensions
$\left(E_{c},\pi\right)$ and $\left(E_{d},\pi\right)$ are not definably equivalent.
\end{proposition}

\begin{proof}
For a contradiction suppose that there is a definable homomorphism  $\varphi:E_{c}\rightarrow E_{d}$ making the diagram below commute:

\begin{displaymath}
\xymatrix{
    0 \ar[r] & \mathbb{R} {\ar[r]_i} {\ar[dr]_i} & E_{c} {\ar[d]_{\varphi}} {\ar[r]_{\pi}} & \mathbb{R}^{2} \ar[r] & 0\\
     &  & E_{d} {\ar[ur]_{\pi}}}.
\end{displaymath}

As
$\pi\circ\varphi\left(\left(\left(1,0\right),0\right)\right)=\left(1,0\right)$
and
$\pi\circ\varphi\left(\left(\left(0,1\right), 0\right)\right)=\left(0,1\right)$,
it follows that
$\varphi\left(\left(\left(1,0\right),0\right)\right)=\left(\left(1,0\right), s\right)$
and
$\varphi\left(\left(\left(0,1\right), 0\right)\right)=\left(\left(0,1\right), t\right)$
for some $t,s\in\mathbb{R}$. Let the definable homomorphisms
 $h_{1},h_{2}:\mathbb{R}\rightarrow E_{d}$ be defined by
$h_{1}\left(x\right)=\left(\left(x,0\right), xs\right)$,
$h_{2}\left(x\right)=\left(\left(0,x\right), xt\right)$. Since the subgroup of
$E_c$ consisting of all elements of the form $\left(\left(y,0\right), 0\right)$ is isomorphic to
$(\mathbb R, +)$, and we have
$h_{1}\left(1\right)=\varphi\left(\left(\left(1,0\right), 0\right)\right)$
and
$h_{2}\left(1\right)=\varphi\left(\left(\left(0,1\right), 0\right)\right)$,
Fact~\ref{P:homen1iguales} implies that
$\varphi\left(\left(\left(x,0\right), 0\right)\right)=\left(\left(x,0\right), xs\right)$
and
$\varphi\left(\left(\left(0,x\right), 0\right)\right)=\left(\left(0,x\right), xt\right)$
for all $x\in\mathbb{R}$.

In $E_{c}$ we have
\[\left(\left(0,0\right), z-cxy\right)\cdot \left(\left(x,0\right), 0\right)\cdot \left(\left(0,y\right), 0\right)
=\left(\left(x,y\right)z\right),\]
which, since $\varphi\left(\left(0,0\right), w\right)=\left(\left(0,0\right), w\right)$ for all $w\in \mathbb R$, implies that

\begin{eqnarray*}
\varphi\left(\left(x,y\right), z\right) &=&
\varphi\left(\left(0,0\right), z-cxy\right)\cdot \varphi\left(\left(x,0\right), 0\right)\cdot \varphi\left(\left(0,y\right), 0\right)  \\
&=&\left(\left(0,0\right), z-cxy,\right)\cdot \left(\left(x,0\right), xs\right)\cdot \left(\left(0,y\right),yt\right) \\
&= &\left(\left(x,y\right), z+xs+yt+xy\left(d-c\right)\right).
\end{eqnarray*}

As
\[
\varphi\left(\left(\left(1,1\right), 1\right)^{2}\right) =
\varphi\left(\left(2,2\right), 2\right) =
\left(\left(2,2\right), 2\left(1+s+t\right)+4\left(d-c\right)\right),
\]
and
\[
\left(\varphi\left(\left(1,1\right), 1\right)\right)^{2} =
\left(\left(1,1\right), 1+s+t+d-c\right)^{2} =
\left(\left(2,2\right), 2\left(1+s+t\right)+2\left(d-c\right)\right),\]
it follows that $d=c$, a contradiction, whence the proof is complete.
\end{proof}

Returning to our original problem of understanding all
ordered groups of dimension~3, we actually know that
\emph{as groups} all the $E_c$ for $c\neq 0$ are isomorphic. (Recall that we
already know that for any order preserving action---as indeed the trivial
one is---the lexicographic order makes any extension an ordered
group.)

\begin{proposition}\label{P:2.4}
Let $\left(U,i,\pi\right)$ be a  representative of a non-zero class in \\ $\textrm{Ext}\left(\mathbb{R}^{2},\mathbb{R}\right)$, the group of extensions of $\mathbb{R}^{2}$ by $\mathbb{R}$, with
$\mathbb{R}$ a definable $\mathbb{R}^{2}$-module under the trivial action. Then the group $U$ is definably  isomorphic to
$\mathcal{G}_{\textrm{Heis}}.$ \end{proposition}
\begin{proof}
As we observed earlier, $H^{2}\left(\mathbb{R}^{2},\mathbb{R}\right)\simeq(\mathbb{R},+)$ carries the structure of an $\mathbb R\/$-vector space.
Since the groups $E_c$ above form a family of nonisomorphic extensions indexed by $c$, it follows that all the possible extensions are equivalent to $E_c$ for some $c\in \mathbb R$.

To prove the proposition, we thus need only show that every $E_c$ is (definably) isomorphic to either
$\mathcal{G}_{\textrm{Heis}}$ or to the abelian group $\mathbb R\times \mathbb R\times \mathbb R$.
Since $E_0$ is isomorphic to $\mathbb R\times \mathbb R\times \mathbb R$, we now fix $c\not=0$.
Let $\sigma:\mathcal{G}_{\textrm{Heis}} \rightarrow E_c$ be given by
\[\left(x,y,z\right) \mapsto
\left(\left(x,y\right),2cz\right).\]
It is evident that $\sigma$ is an isomorphism, completing the proof.
%
\end{proof}

Combining the results above we have

\begin{theorem}\label{T:2.5}
Let $\left(\mathbb{R},+\right)$ be a definable
$\left(\mathbb{R}^{2},+\right)$-module with the trivial action on $\mathcal R$. Then
the definable extensions of $\mathbb{R}^{2}$ by $\mathbb{R}$ are (definably)
isomorphic to either $\mathbb{R}^{3}$ or $\mathcal{G}_{\textrm{Heis}}$.

In particular, let $G$ be an ordered decomposable group of
dimension~3 and $N$ be its (unique) convex o-minimal normal subgroup
such that $H:=G/N$ is isomorphic to $\mathbb R\times \mathbb R$ and
the natural action of $H$ on $N$ is the trivial action. Then $G$ is
(definably) isomorphic (as a group) to either $\mathbb R^3$ or
$\mathcal{G}_{\textrm{Heis}}$, both ordered lexicographically
according to the decomposition above.
\end{theorem}

\begin{proof}
All that requires proof is the assertion that all $E_c$ with $c\neq 0$ as definable ordered groups
are isomorphic; indeed, if the underlying group is $(\mathbb R^3, +)$, the definable ordered group isomorphism is trivially given by the isomorphisms of the convex subgroups.

Any definable order that makes $\mathcal{G}_{\textrm{Heis}}$ into an ordered group must have the subgroup consisting of all matrices
\[
\left[\begin{array}{ccc}

1 & 0 & z\\
0 & 1 & 0\\
0 & 0 &
1\end{array}\right]\]
with $z\in \mathbb R$ as its o-minimal convex subgroup $N$, and the quotient group $\mathcal{G}_{\textrm{Heis}}/N\equiv (R^2,+)$ then can be ordered in any way---all orders give
rise to definable isomorphic ordered groups. Thus, after applying a definable ordered group isomorphism, we may assume that
the ordered group $\mathcal{G}_{\textrm{Heis}}$ is given by all real matrices
\[
\left[\begin{array}{ccc}
1 & x & z\\
0 & 1 & y\\
0 & 0 &
1\end{array}\right]\]
ordered lexicographically with the variables in the order $x, y, z$.

Similarly,
every ordered group
\[E_c:=\{ ((g_1, g_2), m_1)\mid \ m_1,g_1, g_2\in \mathbb R\}
\]
that comes from an extension 
must have $N_c:=\{ ((0,0), m_1)\mid m_1\in \mathbb R\}$ as its convex
o-minimal normal subgroup. Moreover, all orderings of $E_c/N_c$ give
rise to definably isomorphic ordered groups. We thus can assume that
the order on $E_c$ is given by the lexicographic order given by the
variables $g_1, g_2, m_1$ in that order. It now is evident that the
map $\sigma$ defined in the proof of Proposition~\ref{P:2.4} is an
isomorphism of definable groups. As we see in Section~\ref{8.3}, the definable {\em ordered\/} group isomorphism type
varies according to whether $c$ is positive or negative.
\end{proof}

\medskip
\subsection{$H:=\mathbb{R}\times\mathbb{R}$ and the action of
$H$ on $\mathbb{R}$ is nontrivial}\label{7.2}\hspace*{\fill} \\



We begin with an observation. Suppose that $G$ and $G'$ are ordered groups with o-minimal convex normal subgroup $N$, respectively $N'$, such that the quotient groups $H:=G/N$ and $H':=G'/N'$ are isomorphic to $\mathbb R^2$ and the actions
$\gamma\colon H\to \aut (N)$ and $\gamma'\colon H'\to N'$ are nontrivial. Fact \ref{P:accionesequi} implies that there is an isomorphism of modules between
$(H, N, \gamma)$ and $(H', N', \gamma')$. Since any two orders
on $H$ (and $H'$) give rise to isomorphic ordered groups, we can assume that the isomorphism respects the order between the triples.

We now prove

\begin{theorem}\label{T:2.6}
Let $\left(\left(\mathbb{R},+\right),\gamma\right)$ be a definable $\left(\mathbb{R}^{2},+\right)$-module with nontrivial action
$\gamma$. Then
$H^{2}\left(\mathbb{R}^{2},\mathbb{R},\gamma\right)=\left\{
0\right\} $.
\end{theorem}

\begin{proof}
By Fact \ref{P:accionesequi}, the action
$\varphi\left(x,y\right)=e^{x}$ is definably equivalent to every nontrivial action $\gamma$ of
$\mathbb{R}^{2}$  on $\mathbb{R}$, hence
$H^{2}\left(\mathbb{R}^{2},\mathbb{R},\gamma\right)\simeq
H^{2}\left(\mathbb{R}^{2},\mathbb{R},\varphi\right)$. It thus suffices to calculate
$H^{2}\left(\mathbb{R}^{2},\mathbb{R},\varphi\right)$.

Let $K=\left\{ \left(0,y\right)\vert y\in\mathbb{R}\right\}
\trianglelefteq\left(\mathbb{R}^{2},+\right)$.
Observe that $\phi$ makes $(\mathbb R,+)$ into a $K\/$-module under the trivial action. An
application of Lemma~\ref{smaller dimensional groups}(4) yields that $H^2(K,\mathbb{R})=\{0\}$, and hence
$(H^2(K,\mathbb{R}))^{\mathbb{R}^2/K}=\{0\}$. By direct calculation, one can see that
$H^i(K,\mathbb{R})$ is an $\mathbb{R}^2/K$ module under a nontrivial action and
Lemma~\ref{smaller dimensional groups}(2) implies that
$H^1(K,\mathbb{R})\simeq\mathbb{R}$. Since $\mathbb{R}^2/K$ is isomorphic to $\mathbb{R}$, we have
by Lemma~\ref{smaller dimensional groups}(3) that
$H^1(\mathbb{R}^2/K, H^1(K,\mathbb{R}))=\{0\}$.
As $K$ acts trivially on $\mathbb{R}$, we have that $\mathbb{R}^K=\mathbb{R}$ and thus
$H^2(\mathbb{R}^2/K, \mathbb{R}^K)=\{0\}$, again by Lemma~\ref{smaller dimensional groups}(4).
Putting together all of the above, an application of Corollary\ref{all you need}(1) gives
$H^2(\mathbb{R}^2, \mathbb{R}, \phi)=\{0\}$.
\end{proof}


The following now is an immediate consequence of Theorem~\ref{T:2.6} and the discussion preceding it.

\begin{theorem}\label{T:2.7}
Let $\left(\left(\mathbb{R},+\right),\gamma\right)$ and
$\left(\left(\mathbb{R},+\right),\gamma^{\prime}\right)$ be two
definable $\left(\mathbb{R}^{2},+\right)$-modules with nontrivial
actions $\gamma$ and $\gamma^{\prime}$, respectively. Then any two definable groups
$G$ and $G'$ that are extensions of $\mathbb R^2$ by $\mathbb R$
under $\gamma$, respectively $\gamma'$, are definably isomorphic
as groups.
\end{theorem}

In particular, all decomposable three-dimensional ordered groups
definable in $\mathcal R$ such that the quotient $H$ of $G$ by the
unique convex o-minimal subgroup is isomorphic to $\mathbb R\times
\mathbb R$ and the action is nontrivial, are definably
isomorphic (as groups) to $\mathbb{R}\rtimes_{\gamma}\mathbb{R}^{2}$
where $\gamma\left(x,y\right)=e^{x}$ (this is an order preserving
action). In view of this, we shall abuse notation and denote by
$\left(\mathbb{R}\rtimes\mathbb{R}\right)\times\mathbb{R}$ the group
given by the action $\gamma\left(x\right)=e^{x}$.



\medskip
\subsection{$H:=\mathbb{R}\rtimes\mathbb{R}$ and the action of
$H$ on $\mathbb{R}$ is trivial}\label{7.3}\hspace*{\fill} \\

%

We here assume that we have a group that is a
definable extension of $\mathbb{R}\rtimes\mathbb{R}$ by $\mathbb{R}$
with the trivial action.

As in the preceding cases, we first must compute the relevant cohomology groups.

\begin{lemma}\label{L:SDNT}
Let $\left(\mathbb{R},+\right)$ be a definable $\left(\mathbb{R}\rtimes\mathbb{R},\cdot \right)$-module with the trivial action then
$H^{1}\left(\mathbb{R}\rtimes\mathbb{R},\mathbb{R}\right)\simeq\left(\mathbb{R},+\right)$
and $H^{2}\left(\mathbb{R}\rtimes\mathbb{R},\mathbb{R}\right)=\left\{
0\right\}$.
\end{lemma}

\begin{proof}
Let $K:=\left\{ \left(x,0\right)\vert x\in\mathbb{R}\right\}
\trianglelefteq\mathbb{R}\rtimes\mathbb{R}$. Routine calculations (see \cite{BarrigaThesis})
show that
$H^{0}\left(K,\mathbb{R}\right)$ is an $\mathbb{R}\rtimes\mathbb{R}/K$-module with
the trivial action and that $\mathbb{R}\rtimes\mathbb{R}/K$ acts nontrivially on
$H^{1}\left(K,\mathbb{R}\right)$.
The lemma then follows by applying Fact~\ref{C:HyS}, Lemma~\ref{smaller dimensional groups}, and
Corollary~\ref{all you need}. Details are omitted and left to the interested reader.
\end{proof}

With Lemma~\ref{L:SDNT} in hand, we have

\begin{theorem}\label{T:SDNT}
Let $\left(\mathbb{R},+\right)$ be a definable
$\left(\mathbb{R}\rtimes\mathbb{R},\cdot \right)$-module with trivial
action. Then the cartesian product
$\left(\mathbb{R}\times\left(\mathbb{R}\rtimes\mathbb{R}\right)\right)$ is the
only group which is an extension of
$\left(\mathbb{R}\rtimes\mathbb{R},\cdot\right)$ by
$\left(\mathbb{R},+\right)$.

Modulo isomorphism, all such ordered groups are
definably isomorphic to the ordered group $\left(\mathbb{R}\times\left(\mathbb{R}\rtimes\mathbb{R}\right)\right)$ with the lexicographic order
given first by the last coordinate, second by the second coordinate and last by the first coordinate.
\end{theorem}

\begin{proof}
Only the final claim about the ordered groups needs to be checked. We know that any definable ordered group $G$ with o-minimal convex normal subgroup $N$ and $H=G/N$ is ordered lexicographically first by the order in $H$ and lastly by the order in $N$. As we already have proved that any ordered groups isomorphic as groups to $N\simeq (\mathbb{R},+)$ or
$H\simeq (\mathbb R\rtimes \mathbb R, \cdot)$ are isomorphic as definable ordered groups, the assertion is evident.
\end{proof}

\medskip
\subsection{$H:=\mathbb{R}\rtimes\mathbb{R}$ and the action of
$H$ on $\mathbb{R}$ is nontrivial}\label{7.4}\hspace*{\fill} \\

%

The analysis in this case is far more delicate than in the preceding ones.
We first compute the isomorphism classes given by definable actions
of modules.

\begin{lemma}\label{L:3.4.4.1}
Let $\left(\mathbb{R},+\right)$ be a definable
$\mathbb{R}\rtimes\mathbb{R}$-module with nontrivial action
$\gamma$. Then
$\gamma\left(\left(x,y\right)\right)\left(a\right)=e^{cy}a$, for
some $c\not=0$.
\end{lemma}

\begin{proof}
Let $\left(\mathbb{R},\gamma_{1}\right)$ and $\left(\mathbb{R},\gamma_{2}\right)$ be the
definable $\left(\mathbb{R},+\right)$-modules where
$\gamma_{1}\left(x\right)\coloneqq\gamma\left(x,0\right)$ and
$\gamma_{2}\left(y\right)\coloneqq\gamma\left(0,y\right)$ for all $x, y\in\mathbb{R}$. By
Fact~\ref{P:accionesgrales} for $i=1,2$ we have that
$\gamma_{i}\left(x\right)=e^{{c_{i}}x}$ for some $c_{i} \in \mathbb{R}$ . Thus,

\[\gamma\left(\left(x,y\right)\right) =
\gamma\left(\left(x,0\right)\cdot \left(0,y\right)\right) =
\gamma_{1}\left(x\right)\gamma_{2}\left(y\right) =
e^{c_{1}x+c_{2}y} \] and
\[
\gamma\left(\left(1,1\right)\cdot \left(1,1\right)\right) =
\gamma\left(\left(1+e,2\right)\right) =
\left(\gamma\left(1,0\right)\right)^{1+e}\left(\gamma\left(0,1\right)\right)^{2}.\]
As
$(\gamma [(1,1)]^{2} =[\gamma (1,0)]^{2}[\gamma(0,1)]^{2}$ we thus have
$\gamma\left(1,0\right)=1$. Since $\gamma$ is nontrivial
$\gamma\left(0,1\right)\neq1$, hence
$\gamma\left(\left(x,y\right)\right)\left(a\right)=e^{cy}a$ for
some non-zero real constant $c$.
\end{proof}

The analysis now divides according to whether or not $c=1$

Applying Lemma~\ref{smaller dimensional groups} and
Corollary~\ref{all you need} yields

\begin{lemma}
Let $\left(\mathbb{R},+\right)$ be a definable
$\left(\mathbb{R}\rtimes\mathbb{R},\cdot\right)$-module with
nontrivial action $\gamma=e^{cy}$ with $c\neq 1$. Then
$H^{1}\left(\mathbb{R}\rtimes\mathbb{R},\mathbb{R},\gamma\right)=\left\{
0\right\}$ and
$H^{2}\left(\mathbb{R}\rtimes\mathbb{R},\mathbb{R},\gamma\right)=\left\{
0\right\}$.
\end{lemma}

Hence if $c\not=1$ the only group extensions are the semidirect products.

We thus turn to the case $c=1$.  Similar computations yield
$H^{1}\left(\mathbb{R}\rtimes\mathbb{R},\mathbb{R},\gamma\right)\simeq\mathbb{R}$
and
$H^{2}\left(\mathbb{R}\rtimes\mathbb{R},\mathbb{R},\gamma\right)\preceq
\mathbb{R}$.
We prove that if $c=1$ then
$H^{2}\left(\mathbb{R}\rtimes\mathbb{R},\mathbb{R},\gamma\right)\simeq\left(\mathbb{R},+\right)$,
by finding a nontrivial extension.

\begin{definition}
For $k\in\mathbb{R}\setminus\left\{ 0\right\}$, let
$T_{k}$ denote the group defined on the set
$\mathbb{R}\times\left(\mathbb{R}\times\mathbb{R}\right)$ whose
product is given by
\[
\left(x_{1},\left(y_{1},z_{1}\right)\right)\cdot_k \left(x_{2},\left(y_{2},z_{2}\right)\right)=\left(x_{1}+x_{2}e^{z_{1}}+ky_{2}z_{1}e^{z_{1}},\left(y_{1}+y_{2}e^{z_{1}},z_{1}+z_{2}\right)\right).
\]
\end{definition}
Easy calculations show that $T_{k}$ indeed is a group.
We also have

\begin{fact}\label{Tk}
Let $i\left(x\right)=\left(x,\left(0,0\right)\right)$,
$\pi\left(x,\left(y,z\right)\right)=\left(y,z\right)$, and
$k\in\mathbb{R}\setminus\left\{ 0\right\} $. Then
$0\rightarrow\mathbb{R}\overset{i}{\rightarrow}
T_{k}\overset{\pi}{\rightarrow}\mathbb{R}\rtimes\mathbb{R}\rightarrow0$
is a definable extension of $\mathbb{R}\rtimes\mathbb{R}$ by
$\mathbb{R}$ with action
$\gamma\left(\left(x,y\right)\right)\left(a\right)=e^{y}a$.
\end{fact}

\begin{proposition}\label{T:2.4.2}
Let $\left(\mathbb{R},+\right)$ be a definable
$\left(\mathbb{R}\rtimes\mathbb{R},\cdot\right)$-module with action
\\ $\gamma\left(\left(x,y\right)\right)\left(a\right)=e^{y}a$. Then
$H^{2}\left(\mathbb{R}\rtimes\mathbb{R},\mathbb{R},\gamma\right)\simeq\left(\mathbb{R},+\right)$.
\end{proposition}

\begin{proof}
We already have that
$H^{2}\left(\mathbb{R}\rtimes\mathbb{R},\mathbb{R},\gamma\right)\preceq\left(\mathbb{R},+\right)$
and thus is either $\{0\}$ or $\mathbb R$. But by Fact \ref{Tk}, there is a
nontrivial extension so $H^{2}\left(\mathbb{R}\rtimes\mathbb{R},\mathbb{R},\gamma\right)$ must be isomorphic to $\mathbb R$.
\end{proof}

Furthermore, as in Proposition \ref{P:EcEd} one can show that the family
of nontrivial extensions is precisely the family of groups $T_{k}$.

\begin{fact}\label{F:TcTd} Let $c,d\in\mathbb{R}\setminus\left\{ 0\right\} $. If $c\neq d$ then the extensions $\left(T_{c},\pi\right)$ and $\left(T_{d},\pi\right)$ are not definably equivalent.
\end{fact}

As we show below, the groups $T_{k}$ as $k$ ranges over $\mathbb R\setminus\{0\}$ are definably isomorphic.
We thus adopt notation that appears in the literature by renaming $T_1$ as $G_3$.

\begin{lemma}\label{P:2.4.3}
Let $\left(U,i,\pi\right)$ be a  nontrivial extension of
$\mathbb{R}\rtimes\mathbb{R}$ by $\left(\mathbb{R},\gamma\right)$
with action
$\gamma\left(\left(y,z\right)\right)\left(a\right)=e^{z}a$. Then the
group $U$ is definably  isomorphic to $G_3$.
\end{lemma}
\begin{proof}
Since we have that
$H^{2}\left(\mathbb{R}\rtimes\mathbb{R},\mathbb{R},\gamma\right)\simeq\left(\mathbb{R},+\right)$
and also that the groups $T_{k}$ where $k\in\mathbb R\setminus\{0\}$
form a family of nonisomorphic extensions, it
follows that all the possible groups are isomorphic to $T_k$ for some $k$.
To prove the lemma, we only need to show that every $T_k$
is isomorphic to either $G_{3}$ or to the group
$\mathbb{R}\rtimes_{\gamma}\left(\mathbb{R}\rtimes\mathbb{R}\right)$.

Note first that if $k=0$ we immediately have an isomorphism with
$\mathbb{R}\rtimes_{\gamma}\left(\mathbb{R}\rtimes\mathbb{R}\right)$).
Now let $k\not=0$ and let $f$ denote the cocycle
$f\left(\left(y_{1},z_{1}\right),\left(y_{2},z_{2}\right)\right)=y_{2}z_{1}e^{z_{1}}$
of the extension $G_3$.
Then $T_k$ is a group associated with the cocycle $kf$ and its product is given by
\[
\left(x_{1},\left(y_{1},z_{1}\right)\right)\cdot_k\left(x_{2},\left(y_{2},z_{2}\right)\right)
=\left(x_{1}+x_{2}e^{z_1}+kf\left(\left(y_{1},z_{1}\right),\left(y_{2},z_{2}\right)\right),
\left(y_{1}+y_{2}e^{z_1},z_{1}+z_{2}\right)\right).\]
Define $\sigma\colon T_{k} \to G_{3}$ by
\begin{equation*}
\left(x,\left(y,z\right)\right) \mapsto \left(\frac{1}{k}x,\left(y,z\right)\right).
\end{equation*}
We prove that $\sigma$ is an isomorphism. Indeed,

\begin{eqnarray*}
\sigma\left(x_{1},\left(y_{1},z_{1}\right)\right)& \cdot_k &\left(x_{2},\left(y_{2},z_{2}\right)\right)\\
& = &
\sigma\left(x_{1}+x_{2}e^{z_1}+kf\left(\left(y_{1},z_{1}\right),\left(y_{2},z_{2}\right)\right),
\left(y_{1}+y_{2}e^{z_1},z_{1}+z_{2}\right)\right)\\
& = &
\left(\frac{1}{k}\left(x_{1}+x_{2}e^{z_1}\right)+f\left(\left(y_{1},z_{1}\right),\left(y_{2},z_{2}\right)\right),
\left(y_{1}+y_{2}e^{z_1},z_{1}+z_{2}\right)\right)
\end{eqnarray*}
and
\begin{eqnarray*}
\sigma\left(x_{1},\left(y_{1},z_{1}\right)\right) & \cdot_1 & \sigma\left(x_{2},\left(y_{2},z_{2}\right)\right)\\
& = &
\left(\frac{1}{k}x_{1},\left(y_{1},z_{1}\right)\right) \cdot_1 \left(\frac{1}{k}x_{2},\left(y_{2},z_{2}\right)\right)\\
& = &
\left(\frac{1}{k}x_{1}+\frac{1}{k}x_{2}e^{z_1}+f\left(\left(y_{1},z_{1}\right),\left(y_{2},z_{2}\right)\right),
\left(y_{1}+y_{2}e{z_1},z_{1}+z_{2}\right)\right)
\end{eqnarray*}
as required.
\end{proof}

We now have

\begin{theorem}\label{T:2.4.4}
Let $\left(\mathbb{R},+\right)$ be a definable
$\left(\mathbb{R}\rtimes\mathbb{R},\cdot\right)$-module with action
\\ $\gamma\left(\left(x,y\right)\right)\left(a\right)=e^{y}a$. Then the
groups which are definable extensions of
$\mathbb{R}\rtimes\mathbb{R}$ by $\mathbb{R}$ are isomorphic to
either $G_{3}$ or
$\mathbb{R}\rtimes_{\gamma}\left(\mathbb{R}\rtimes\mathbb{R}\right)$.
\end{theorem}

Unlike the previous cases the situation here is more subtle and thus additional
work is required to determine the possible isomorphism classes---as groups---of (orderable) decomposable
definable groups that arise. We begin with some easy
observations.

\begin{proposition}\label{P:3.1} Let $\gamma$ be the nontrivial action
of $\mathbb{R}\rtimes\mathbb{R}$ on $\left(\mathbb{R},+\right)$ given by
$\gamma\left(\left(x,y\right)\right)\left(a\right)=e^{cy}a$, where $c\not=0$.
Then $\mathbb{R}\rtimes_{\gamma}\left(\mathbb{R}\rtimes\mathbb{R}\right)$
is definably isomorphic to the group
$\mathbb{R}^{2}\rtimes_{\tau_c}\mathbb{R}$ where
\[\tau_c\left(z\right)\left(x,y\right)\coloneqq\left[\begin{array}{cc}
e^{cz} & 0\\
0 & e^{z}\end{array}\right]\left(\begin{array}{c}
x\\
y\end{array}\right).\]
Furthermore, the action is order-preserving if and only
if $c>0$.
\end{proposition}

\begin{proof}
The mapping
$\left(x,\left(y,z\right)\right) \mapsto
\left(\left(x,y\right),z\right)$ is an isomorphism.
\end{proof}

We now determine which of the groups
$\mathbb{R}^{2}\rtimes_{\tau_c}\mathbb{R}$ are definably isomorphic, where the
action $\tau_c$ is as displayed above.

\begin{proposition}\label{P:3.2}
Let $c\in\mathbb{R}\setminus\left\{ 0\right\} $. Then
$\mathbb{R}^{2}\rtimes_{\tau_{c}}\mathbb{R}$ and
$\mathbb{R}^{2}\rtimes_{\tau_{1/c}}\mathbb{R}$ are
definably isomorphic.
\end{proposition}
\begin{proof}
Let $\varphi:\mathbb{R}^{2}\rtimes_{\tau_{c}}\mathbb{R} \to
\mathbb{R}^{2}\rtimes_{\tau_{1/c}}\mathbb{R}$ be the definable bijection
given by $((x,y),z) \stackrel{\varphi}{\mapsto} ((y,x),cz)$. As
\begin{eqnarray*}
\varphi\left(\left(\left(x,y\right),z\right)\cdot \left(\left(x^{\prime},y^{\prime}\right),z^{\prime}\right)\right) & = & \varphi\left(\left(x+e^{cz}x^{\prime},y+e^{z}y^{\prime}\right),z+z^{\prime}\right)\\
 & = & \left(\left(y+e^{z}y^{\prime},x+e^{cz}x^{\prime}\right),c\left(z+z^{\prime}\right)\right)\\
& = & \left(\left(y,x\right),cz\right)\cdot \left(\left(y^{\prime},x^{\prime}\right),cz^{\prime}\right)\\
 & = &  \varphi\left(\left(x,y\right),z\right)\cdot \varphi\left(\left(x^{\prime},y^{\prime}\right),z^{\prime}\right).
\end{eqnarray*}
we have that $\varphi$ is an isomorphism.
\end{proof}

The next observation follows
easily by direct calculation.

\begin{proposition}\label{F:3.1}
Let $d\in\mathbb{R}\setminus\left\{ 0\right\} $ and
$\left(\left(x,y\right),z\right)\in \mathbb{R}^{2}\rtimes_{\tau_{d}}\mathbb{R}$ be fixed but arbitrary.
Then
    \begin{enumerate}
\item If $z\neq0$, the function  $h_{z}:\mathbb{R} \rightarrow
\mathbb{R}^{2}\rtimes_{\tau_{d}}\mathbb{R}$ defined by

\[w \mapsto \left(\left(\left(\frac{e^{dzw}-1}{e^{dz}-1}\right)x,\left(\frac{e^{zw}-1}{e^{z}-1}\right)y\right),wz\right)\]
is a definable homomorphism.

\item If $z=0$, the function $h_{0}:\mathbb{R}\rightarrow \mathbb{R}^{2}\rtimes_{\tau_{d}}\mathbb{R}$
given by
\[w \mapsto \left(\left(wx,wy\right),0\right)\]
is a definable homomorphism.
  \end{enumerate}
\end{proposition}

With the way now paved we prove

\begin{lemma}\label{P:3.3}
Let $c,d\in\mathbb{R}\setminus\left\{ 0\right\} $ with
$c\notin\left\{d, 1/d\right\} $. Then
$\mathbb{R}^{2}\rtimes_{\tau_{c}}\mathbb{R}$ and
$\mathbb{R}^{2}\rtimes_{\tau_{d}}\mathbb{R}$ are not definably isomorphic.
\end{lemma}

\begin{proof}
For a contradiction suppose there is a definable isomorphism
$\varphi:\mathbb{R}^{2}\rtimes_{\tau_{c}}\mathbb{R}\rightarrow
\mathbb{R}^{2}\rtimes_{\tau_{d}}\mathbb{R}$. For $i\in\left\{ 1,2,3\right\}$ let
$\theta_{i}:\mathbb{R}\rightarrow\mathbb{R}^{2}\rtimes_{\tau_{d}}\mathbb{R}$ given by
$\theta_{1}\left(x\right)=\varphi\left(\left(x,0\right),0\right)$,
$\theta_{2}\left(y\right)=\varphi\left(\left(0,y\right),0\right)$ and,
$\theta_{3}\left(z\right)=\varphi\left(\left(0,0\right),z\right)$.
It is evident that each $\theta_{i}$ is a definable homomorphism and
$\varphi\left(\left(x,y\right),z\right)  =
\theta_{1}\left(x\right)\theta_{2}\left(y\right)\theta_{3}\left(z\right)$.

For $i\in\left\{ 1,2,3\right\}$ put
$\left(\left(x_{i},y_{i}\right),z_{i}\right):={\theta_{i}}\left(1\right)$
and define the functions $h_{i,z_{i}}: \mathbb{R}\to \mathbb{R}^{2}\rtimes_{\tau_{d}}\mathbb{R}$ by

\[ {h_{i,z_{i}}}\left(x\right) =
\begin{cases}
\left(\left(\left(\frac{e^{dz_{i}x}-1}{e^{dz_{i}}-1}\right)x_{i},\left(\frac{e^{z_{i}x}-1}{e^{z_{i}}-1}\right)y_{i}\right),xz_{i}\right) &
\text{if $z_{i}\neq0$} \\
\noalign{\medskip}
\left(\left(xx_{i},xy_{i}\right),0\right) & \text{if $z_{i}=0$.}
\end{cases}
\]
\smallskip

\noindent These are definable homomorphisms by Proposition~\ref{F:3.1}. Since
$h_{i,z_{i}}\left(1\right)=\theta_{i}\left(1\right)$ for all $z_{i}$ and $i\in\left\{ 1,2,3\right\}$,
Fact~\ref{P:homen1iguales} implies that $h_{i,z_{i}}=\theta_{i}$
for all $i\in\left\{ 1,2,3\right\}$.

We assert that $z_1=z_2=0$. Indeed, suppose $z_1\not=0$. Let
${\pi_{3}}:\mathbb{R}^{2}\rtimes_{\tau_{d}}\mathbb{R}\rightarrow\mathbb{R}$
denote the projection homomorphism onto the third coordinate.
Then
\[
\pi_{3}\left(\varphi\left(\left(\left(1,0\right),1\right)^{2}\right)\right)
=
\left(1+e^{c}\right)z_{1}+2z_{3}\]

and

\[\pi_{3}\left(\left(\varphi\left(\left(1,0\right),1\right)\right)^{2}\right)
=  2z_{1}+2z_{3},
\]
from which it follows that $z_{1} = 0$, contradiction.
A similar argument, replacing $\left(\left(1,0\right),1\right)$ by $\left(\left(0,1\right),1\right)$, shows that $z_2=0$.
%
Hence
$\theta_{1}\left(x\right)=\left(\left(xx_{1},xy_{1}\right),0\right)$
and
$\theta_{2}\left(x\right)=\left(\left(xx_{2},xy_{2}\right),0\right)$.

We next show that $z_{3}\notin\left\{ 0,1\right\}$.
Direct computation yields that
\[
\varphi\left(\left(\left(1,0\right),1\right)^{2}\right) =
\left(\left(\left(1+e^{c}\right)x_{1}+\left(e^{dz_{3}}+
1\right)x_{3},\left(1+e^{c}\right)y_{1}+\left(1+e^{z_{3}}\right)y_{3}\right),2z_{3}\right),\]
and
\[\left(\varphi\left(\left(1,0\right),1\right)\right)^{2} =
\left(\left(\left(x_{1}+x_{3}\right)\left(e^{dz_{3}}+1\right),\left(y_{1}+y_{3}\right)
\left(1+e^{z_{3}}\right)\right),2z_{3}\right),\]
from which we obtain
\begin{eqnarray}\label{Ec:1.1} \left(1+e^{c}\right)x_{1}  =
\left(e^{dz_{3}}+1\right)x_{1}
\end{eqnarray}
and
\begin{eqnarray}\label{Ec:1.2}
\left(1+e^{c}\right)y_{1} = \left(e^{z_{3}}+1\right)y_{1}.
\end{eqnarray}
If $z_3=0$, these force $x_1=y_1=0$, which, together with $z_1=0$, gives
that $\varphi\left(\left(x,y\right),z\right)=\theta_{2}\left(y\right)\theta_{3}\left(z\right)$.
As this implies that $\varphi$ is not injective, we must have that $z_3\not=0$, and, as a
consequence that
\[
    \theta_{3}\left(x\right)=\left(\left(\left(\frac{e^{dz_{3}x}-1}{e^{dz_{3}}-1}\right)x_{3},\left(\frac{e^{z_{3}x}-1}{e^{z_{3}}-1}\right)y_{3}\right),xz_{3}\right).\]

We also have
\[\varphi\left(\left(\left(0,1\right),1\right)^{2}\right) =
\left(\left(\left(1+e\right)x_{2}+\left(e^{dz_{3}}+1\right)x_{3},
\left(1+e\right)y_{2}+\left(e^{z_{3}}+1\right)y_{3}\right),2z_{3}\right)
\]
and
\[\left(\varphi\left(\left(0,1\right),1\right)\right)^{2} =
 \left(\left(\left(x_{2}+x_{3}\right)\left(e^{dz_{3}}+1\right),\left(y_{2}+
 y_{3}\right)\left(1+e^{z_{3}}\right)\right),2z_{3}\right)
 \]
which in turn yield
\begin{eqnarray}\label{Ec:2.1}
\left(1+e\right)x_{2} = \left(e^{dz_{3}}+1\right)x_{2}
\end{eqnarray}
and
\begin{eqnarray}\label{Ec:2.2}
\left(1+e\right)y_{2} = \left(e^{z_{3}}+1\right)y_{2}.
\end{eqnarray}
Suppose now that $z_{3}=1$. Equation \eqref{Ec:1.1}~implies that $x_{1}=0$, and  Equation~\eqref{Ec:1.2} that $c=1$ or $y_{1}=0$. With $y_1=0$ we again find that
$\varphi\left(\left(x,y\right),z\right)=\theta_{2}\left(y\right)\theta_{3}\left(z\right)$ and thus we
are left with $c=1$. Equation \eqref{Ec:2.1} implies that $d=1$ or $x_{2}=0$, and as we have assumed that $c\not=d$ we have $x_{2}=0$. Combining $z_3=1$ with $x_1=x_2=z_1=z_2=0$ yields
 \[
        \varphi\left(\left(x,y\right),z\right)=\left(\left(\left(\frac{e^{dz}-1}{e^{d}-1}\right)x_{3},xy_{1}+yy_{2}+\left(\frac{e^{z}-1}{e-1}\right)y_{3}\right),zz_{3}\right).\]
From this we have $\varphi\left(\left(y_{2},-y_{1}\right),0\right)=0$, which forces $y_{1}=y_{2}=0$. Then
\\ $\varphi\left(\left(x,y\right),z\right)=\theta_{3}\left(z\right)$, which is impossible. Thus $z_3\not=1$.

By Equation \eqref{Ec:2.2}, $y_{2}=0$ and thus $x_{2}\neq0$, as otherwise $\varphi$ cannot be injective. Equation \eqref{Ec:2.1} with $x_{2}\neq0$ implies that $1=dz_{3}$ and hence
by~\eqref{Ec:1.1} that $c=1$ or $x_{1}=0$. With $c=1$ Equation~\eqref{Ec:1.2} gives $y_1=0$ and
hence
\[
   \varphi\left(\left(x,y\right),z\right)=\left(\left(xx_{1}+yx_{2}
   +\left(\frac{e^{dz_3z}-1}{e^{dz_3}-1}\right)x_{3},\left(\frac{e^{z_3z}-1}{e^{z_3}-1}\right)y_{3}\right),
   zz_3\right).\]
Then  $\varphi\left(\left(-x_{2},x_{1}\right),0\right)=\left(\left(0,0\right),0\right)$ and as $x_{2}\neq0$
again we contradict that $\varphi$ is injective. So we are left with the possibility that $c\not=1$ and $x_{1}=0$. For $\varphi$ to be injective we must have $y_1\not=0$. Invoking Equation~\eqref{Ec:1.2}
we obtain $z_3=c$. Combined with $1=dz_{3}$ this yields $c=1/d$, contradicting the hypothesis that $c\not= d,1/d$. So $\varphi$ cannot exist and the proof is complete.
\end{proof}

We now can state and prove the main result in the case that $H=\mathbb{R}\rtimes\mathbb{R}$ and the action of
$H$ on $\mathbb{R}$ is nontrivial.

\begin{theorem}
Let $G$ be an ordered group with o-minimal convex subgroup $N$ such
that $H:=G/N\equiv (\mathbb R\rtimes \mathbb R)$ and the action of
$H$ on $N$ is nontrivial. Then $G$ is isomorphic to either
$G_{3}$ or $\mathbb R^2 \rtimes_{\tau_c} \mathbb R$ for some
$c\in\left[-1,1\right]$, where
$$\tau_{c}\left(z\right)\left(x,y\right)\coloneqq\left[\begin{array}{cc}
e^{cz} & 0\\
0 & e^{z}\end{array}\right]\left(\begin{array}{c}
x\\
y\end{array}\right).$$ There is one definable
isomorphism class of ordered groups for each $c\in\left[-1,1\right]$ and one for
$G_{3}$.

\end{theorem}

\begin{proof}
Let $N$ be the o-minimal subgroup of $G$ and let $N^*$ be the
pullback of the o-minimal subgroup $N_H$ of $H$ under the quotient
map from $G$ to $H$ Thus $N^*$ is a 2-dimensional ordered convex
subgroup of $G$. Lemma \ref{L:3.4.4.1} implies that $N^*$ is isomorphic to
$\mathbb R^2$ and by construction it is ordered lexicographically by
$N_2/N$ first and then by $N$.

It also shows that, after composing by an (ordered group)
automorphism of $H$ so that the action of $N_H$ on $H/N_H$ is
given by the exponential function,  the action of $G/N^*$ on  $N^*$
is given by
$$\tau_{c}\left(z\right)\left(x,y\right)\coloneqq\left[\begin{array}{cc}
e^{cz} & 0\\
0 & e^{z}\end{array}\right]\left(\begin{array}{c}
x\\
y\end{array}\right)$$ for some $c\in\left[-1,1\right]$.

From Propositions~\ref{P:3.1}, \ref{P:3.2}, and \ref{P:3.3}
it follows that as groups there is one isomorphism class for
every $c\in\left[-1,1\right]$.

To complete the proof of the theorem, we will show that $G_{3}$ and
$\mathbb{R}^{2}\rtimes_{\tau_{c}}\mathbb{R}$ are not definably
isomorphic. For a contradiction, suppose that
$\varphi:\mathbb{R}^{2}\rtimes_{\tau_{c}}\mathbb{R}\rightarrow
G_{3}$ were a definable isomorphism of groups.

Since the only proper nontrivial normal subgroups of
$\mathbb{R}^{2}\rtimes_{\tau_{c}}\mathbb{R}$ and $G_{3}$ are, respectively,
$\mathbb{R}^{2}\times\left\{ 0\right\}$ and $\mathbb{R}\times\left( \mathbb{R}\times\left\{0\right\}\right)$, the restriction of
$\varphi$ to $\mathbb{R}^{2}\times\left\{ 0\right\}$ must be an isomorphism between these normal subgroups. Then,
for all $x,y\in\mathbb{R}$ we have that
$$\varphi\left(\left(x,y\right),0\right)=\left(\left(ax+by,ky\right),0\right)$$
for some $a,b,k\in\mathbb{R}$ such that $ak\neq0$.

The restriction of $\varphi$ to $\left\{ \left(0,0\right)\right\}
\times\mathbb{R}$ also is a definable group isomorphism. Hence there are definable $\theta_1$, $\theta_2$, and $\theta_3$
such that
$\varphi\left(\left(0,0\right),z\right)=\left(\theta_{1}\left(z\right),\theta_{2}\left(z\right),\theta_{3}\left(z\right)\right)$
for all $z\in \mathbb R$.
Since
\begin{eqnarray*}
& & \left(\theta_{1}\left(z+z^{\prime}\right),\left(\theta_{2}\left(z+z^{\prime}\right)\right),\theta_{3}\left(z+z^{\prime}\right)\right)\\
& = & \left(\theta_{1}\left(z\right),\left(\theta_{2}\left(z\right)\right) , \theta_{3}\left(z\right)\right)  \cdot
\left(\theta_{1}\left(z^{\prime}\right),\left(\theta_{2}\left(z^{\prime}\right)\right),\theta_{3}\left(z^{\prime}\right)\right)\\
& = & \left(\theta_{1}\left(z\right)+e^{\theta_{3}\left(z\right)}\theta_{1}\left(z^{\prime}\right)+\theta_{3}\left(z\right)\theta_{2}\left(z^{\prime}\right)e^{\theta_{3}\left(z\right)},\right.\\
& & \;\left.\left(\theta_{2}\left(z\right)+\theta_{2}\left(z^{\prime}\right)e^{\theta_{3}\left(z\right)},\theta_{3}\left(z\right)+\theta_{3}\left(z^{\prime}\right)\right)\right)
\end{eqnarray*}
we immediately see that there is some $k_{3}\in\mathbb{R}$ for which $\theta_{3}\left(z\right)=k_{3}z$ for all
$z\in\mathbb{R}$. We next observe that $\theta_{2}$ is a `crossed homomorphism', and as
$H^{1}\left(\mathbb{R},\mathbb{R},e^{k_{3}z}\right)=\left\{ 0\right\} $, an easy calculation yields  that
there is some $k_{2}\in\mathbb{R}$ such that
$\theta_{2}\left(z\right)=k_{2}\left(e^{k_{3}z}-1\right)$ for all $z\in\mathbb{R}$.

In $\mathbb{R}^{2}\rtimes_{\tau_{c}}\mathbb{R}$ we have
$\left(\left(x,y\right),0\right)\cdot \left(\left(0,0\right),z\right)=\left(\left(x,y\right),z\right)$ and so

\begin{eqnarray*}
\varphi\left(\left(\left(x,y\right),z\right)\right)
& =& \varphi\left(\left(\left(x,y\right),0\right)\right)\cdot \varphi\left(\left(\left(0,0\right),z\right)\right)\\
& = & \left(\left(ax+by,ky\right),0\right)\cdot \left(\left(\theta_{1}\left(z\right),k_{2}\left(e^{k_{3}z}-1\right)\right),k_{3}z\right) \\
& = & \left(\left(ax+by+\theta_{1}\left(z\right),ky+k_{2}\left(e^{k_{3}z}-1\right)\right),k_{3}z\right).
\end{eqnarray*}

Thus

\begin{eqnarray*}
& & \varphi\left(\left(\left(x,y\right),z\right)\right)\cdot
\varphi\left(\left(\left(x^{\prime},y^{\prime}\right),z^{\prime}\right)\right)\\
& =& \left.\left(ax+by+\theta_{1}\left(z\right)+e^{k_{3}z}\left(ax^{\prime}+by^{\prime}+\theta_{1}\left(z^{\prime}\right)\right)
+k_{3}ze^{k_{3}z}\left(ky^{\prime}+k_{2}\left(e^{k_{3}z^{\prime}}-1\right)\right),\right.\right. \\
& & \left. \left(ky+k_{2}\left(e^{k_{3}z}-1\right)+e^{k_{3}z}\left(ky^{\prime}+k_{2}\left(e^{k_{3}z^{\prime}}-1\right)\right),k_{3}\left(z+z^{\prime}\right)\right)\right)
\end{eqnarray*}

and

\begin{eqnarray*}
\varphi\left(\left(\left(x,y\right),z\right)\cdot \left(\left(x^{\prime},y^{\prime}\right),z^{\prime}\right)\right)
& =& \left(a\left(x+x^{\prime}e^{cz}\right)+b\left(y+y^{\prime}e^{z}\right)+\theta_{1}\left(z+z^{\prime}\right),\right.\\
& &\; \left.\left(k\left(y+y^{\prime}e^{z}\right)+k_{2}\left(e^{k_{3}\left(z+z^{\prime}\right)}-1\right),k_{3}\left(z+z^{\prime}\right)\right)\right).
\end{eqnarray*}

Equating the second components yields $ky^{\prime}e^{k_{3}z}=ky^{\prime}e^{z}$; hence as $ak\not=0$ we have $k_{3}=1$.
Equating the first components and putting $z^{\prime}=y^{\prime}=0$ gives $ax^{\prime}e^{cz}=ax^{\prime}e^{cz}$ and thus $c=1$. Then, if $c\neq1$ such
isomorphism can not exist. Next, if we put $z^{\prime}=0$ and from the first components we obtain
$zky^{\prime}e^{z}=0$ for all $z,y^{\prime}\in\mathbb{R}$, which forces $k=0$, a contradiction. The theorem now follows.
\end{proof}

\section{decomposable ordered groups modulo ordered group
isomorphism}\label{sectionOrdered}

In the preceding sections we have analyzed decomposable ordered groups in dimensions~2 and~3 modulo
definable group isomorphism. As noted in the
introduction, even in dimension~2---see~\ref{ordered dimension 2} below---this is not the same as
identifying the classes of decomposable ordered groups modulo definable
\emph{ordered group isomorphism}. We here determine the isomorphism classes taking
into account the ordered group structure. In dimension~3, the analysis follows the case distinctions
of Section~\ref{dimension3}.

Observe that if $G$ is a
definable ordered group and $N$ a normal o-minimal convex subgroup, then $G$
is an extension of $G/N$ by $N$ and the order is given lexicographically. This implies
in particular that if two extensions are isomorphic, then they are isomorphic as ordered groups.

\subsection{Dimension 2}\label{ordered dimension 2}
In dimension two there are three definable ordered groups modulo
order group isomorphism. These are $\mathbb R\ltimes_\gamma \mathbb R$ with
$\gamma=e^x$, and $\gamma=e^{-x}$, and the direct product ($\gamma
=1$), as we now show.

Given any other $\mathbb R\ltimes_\gamma \mathbb R$ where
$\gamma=e^{cx}$, then the group isomorphism $(x,y)\mapsto (|c|x, y)$ is a
definable ordered group isomorphism with one of the three groups above.

It remains to prove that the groups $\mathbb R\ltimes_\gamma \mathbb R$
and $\mathbb R\ltimes_{\gamma'} \mathbb R$, where $\gamma=e^x$ and $\gamma=e^{-x}$,
although isomorphic as definable groups are not isomorphic
as ordered groups. Indeed, for positive $g_1, g_2$ we have
$g_1g_2g_1^{-1}\geq g_2$ in the first case, whereas in the second
case $g_1g_2g_1^{-1}\leq g_2$, and since the groups are non abelian they
are not isomorphic as ordered groups.

\subsection{Extensions of $\mathbb R\times \mathbb R$ by $\mathbb R$ with trivial action.}\label{8.2}

All extensions of $\mathbb R\times \mathbb R$ by $\mathbb R$ with trivial action
are isomorphic as definable extensions---hence as definable ordered groups---to either
$\mathbb R\times \mathbb R\times \mathbb R$ or to one of the groups $E_c$ for $c\in \mathbb R\setminus \{0\}$ as
defined in \ref{HeisSect}. The map $(x,y,z)\mapsto (|c|x, y, z)$ is a
definable order preserving group isomorphism from $E_c$ to
either $E_1$ or $E_{-1}$, depending on whether $c$ is
positive or negative. The following assertion thus completes the classification
of all ordered group isomorphism types that arise in this case.

\begin{claim}
$E_{-1}$ and $E_1$ are not isomorphic as ordered groups.
\end{claim}

\begin{proof} Let $G$ be either $E_1$ or $E_{-1}$ and
let $g:=(a,b,c)$ be any element with $a>0$. (Note that this can be done without
mentioning how the group is representated: just take any representative of a positive
class not equivalent to the identity in $G/N^*$ where $N^*$ is the
unique convex subgroup of dimension 2.)  Next let $h:=(0,k,l)$ be any
element with $k>0$ (that is, any positive element with respect to the order on $G$
in the identity class of $G/N^*$). Easy
computations show that for any such $g$ and $h$, if $G=E_1$ then
\[
ghg^{-1}h^{-1}=(0,0,2ak)>(0,0,0),
\]
and if $G=E_{-1}$, then
\[
ghg^{-1}h^{-1}=(0,0,-2ak)<(0,0,0).\] \end{proof}

Hence in this case, we have three definable ordered groups modulo
definable ordered group isomorphism, namely $\mathbb R^3, E_1$, and $E_{-1}$.

\subsection{Extensions of  $\mathbb R\times \mathbb R$ by $\mathbb R$ with nontrivial action.}\label{8.3}
From the analysis in~\ref{7.2} the situation is as follows.
Put $G:=\{(x,y,z)\mid x,y,z\in \mathbb R\}$ and $N:=\{(0,0,z)\mid z\in \mathbb R\}$.
we have that $G$ is a lexicographically ordered group---ordered first by $x$, second by $y$, and
third by $z$---with an action of $G/N$ on the convex o-minimal subgroup $N$
given by $\gamma_{c,d}(x,y)=e^{cx+dy}$.
For ease of notation, throughout this
subsection we refer to this group as $G_{c,d}$.

The following hold via easy calculations:

\begin{itemize}
\item[a.] If $d>0$ then the map $(x,y,z)\mapsto (x,cx+dy, z)$ is a definable ordered group isomorphism from $G_{c,d}$ onto
$G_{0,1}$;

\item[b.] If $d<0$ then the map $(x,y,z)\mapsto (x,-cx-dy, z)$ is a definable ordered group isomorphism from $G_{c,d}$ onto
$G_{0,-1}$;

\item[c.] If $d=0$ and $c>0$ then the map $(x,y,z)\mapsto (cx, y, z)$ is a definable ordered group isomorphism from $G_{c,d}$ onto $G_{1,0}$;

\item[d.] If $d=0$ and $c<0$ then the map $(x,y,z)\mapsto (|c|x, y, z)$ is a definable ordered group isomorphism from $G_{c,d}$ onto $G_{-1,0}$.
 \end{itemize}

Thus every extension of $\mathbb R\times \mathbb R$ by
$\mathbb R$ with nontrivial action and lexicographically ordered is definably isomorphic
as an ordered group to one of $G_{0,1}$, $G_{0,-1}$, $G_{1,0}$, or $G_{-1,0}$.
To complete the analysis in this case, we assert:

\begin{claim}
The ordered groups $G_{0,1}$, $G_{0,-1}$, $G_{1,0}$, and $G_{-1,0}$ are nonisomorphic (as ordered groups).
\end{claim}

\begin{proof}
In $G_{1,0}$ and $G_{-1,0}$ there is an abelian convex two dimensional subgroup, which is not the case for
$G_{0,1}$ and $G_{0,-1}$.

To show that $G_{1,0}$ and $G_{-1,0}$ are not isomorphic as ordered groups, the same analysis as in the two dimensional case applies. In $G_{1,0}$ let $h$ be an element in
the convex o-minimal subgroup and let $g$ be a positive element. Then $ghg^{-1}\geq h$. If we take any such two elements $h$ and $g$ in  $G_{-1,0}$ then $ghg^{-1}\geq h$. A similar argument shows that the ordered groups $G_{0,1}$ and $G_{0,-1}$ are not isomorphic as ordered groups.
\end{proof}

Hence $G_{0,1}$, $G_{0,-1}$, $G_{1,0}$, and $G_{-1,0}$ are, modulo ordered group isomorphism, the only ordered groups that decompose as an extension of $\mathbb R\times \mathbb R$ by $\mathbb R$ with nontrivial action.

\subsection{Extensions of  $\mathbb R\ltimes_c \mathbb R$ by $\mathbb R$ with trivial action.}\label{8.4}

From the analysis in \ref{7.3}, we know that every such group extension (ordered lexicographically)
is isomorphic to $\mathbb R\ltimes_c \mathbb R\times \mathbb R$.
The only remaining issue in this case is precisely the same
as in the two dimensional case: there are two
ordered groups modulo definable order isomorphism, depending on whether
the action by a positive element is ``expanding'' ($c=1$) or
``contracting'' ($c=-1$). This completes the analysis in this case.

\smallskip

For extensions of $\mathbb R\ltimes_c \mathbb R$ by $\mathbb R$ with nontrivial action (see~\ref{7.4}), the analysis now divides according to whether or not the extension is trivial.

\subsection{Trivial extensions of  $\mathbb R\ltimes_c \mathbb R$ by $\mathbb R$ with nontrivial action.}\label{Subsectionfinal}

From the results in~\ref{7.4}, we can represent such a group as consisting of all elements $(x,y,z)\in \mathbb R^3$ ordered lexicographically by $x$, $y$ and $z$
in this order and such that for all $g=(x,y,z)$ and
$h=(0,l,m)$ we have $ghg^{-1}=(0,e^{cx}l, e^{dx}m)$ for constants
$c$ and $d$. For ease of notation in this subsection we use  $K_{c,d}$ to refer to this ordered group.
We then let $N:=\{(0,0,z)\mid z\in \mathbb R\}$ and $N':=\{(0,y,z)\mid y,z\in \mathbb R\}$ denote, respectively,
the one and two dimensional convex normal subgroups of $K_{c,d}$.

An easy calculation shows that $(x,y,z)\mapsto ((|c|)x,y,z)$
is a definable order preserving group isomorphism from $K_{c,d}$
to either $K_{1,d/c}$ or $K_{-1,-d/c}$ depending on whether $c$ is
positive or negative.

As in the preceding case one can show that for all $f$ and $f'$ the ordered groups $K_{1,f}$ and
$K_{-1,f'}$ are not isomorphic. Indeed, in $K_{1,f}$ the action of a positive class
of $G/N'$ on a positive class of $N'/N$ is a map that is
greater than the identity map on every element of the domain, whereas for $K_{-1,f'}$ the corresponding map
is less than the identity map for all elements of the domain. By Proposition \ref{P:3.2} we also know that if
$f'\neq 1/f$ then $K_{1,f}$ and $K_{1,f'}$ are not even isomorphic as groups.
Thus it remains to determine whether or not for a given value $f$
the groups $K_{1,f}$ and $K_{1,1/f}$ are isomorphic as definable
ordered groups, and likewise for each pair $K_{-1,f}$ and $K_{-1,1/f}$. We have

\begin{claim}
For $f\not=1$ the groups $K_{1,f}$ and $K_{1,1/f}$ are not isomorphic as ordered groups, and likewise for
$K_{-1,f}$ and $K_{-1,1/f}$.
\end{claim}

\begin{proof}

An ordered group isomorphism must in particular induce isomorphisms
between the two convex normal subgroups of dimension~2, which are isomorphic to
$\mathbb R \times \mathbb R$, and also between the quotient of the
groups by these subgroups, isomorphisms which must preserve the group action.
Let $\gamma$ be the action in $K_{1,f}$ of the quotient
$\mathbb R_x$ on the dimension two convex normal subgroup $\mathbb R_y\times \mathbb R_z$. Then
for any nonzero element $x\in \mathbb R_x$ the action $\gamma(x)$ is an isomorphism of the vector space
$\mathbb R_y\times \mathbb R_z$ that has $\mathbb R_y\times \{0\}$ and $\{0\}\times \mathbb R_z$ as
its eigenspaces. The same holds for $K_{1,1/f}$.

A group isomorphism must map eigenspaces to eigenspaces,
and in the ordered group case the coordinates cannot be switched since the eigenspace
$\{0\}\times \mathbb R_z$ is the normal convex o-minimal subgroup.
Hence, a definable ordered group isomorphism
between $K_{1,f}$ and $K_{1,1/f}$ must be multiplication by
a constant in each coordinate. It then is easy to check that there is no constant $k$ in the $x$ variable that
can send the group action $e^x$ in the $y$-coordinate to itself and send the group action
$e^{fx}$ to $e^{x/f}$ in the $z$-cooordinate. A similar argument shows that the groups  $K_{-1,f}$ and $K_{-1,1/f}$ with $f\not=1$ are not isomorphic as ordered groups.

\end{proof}

Thus, modulo definable ordered group isomorphism, the trivial extensions of  $\mathbb R\ltimes_c \mathbb R$ by $\mathbb R$ with nontrivial action are the groups $K_{1,f}$ and $K_{-1,f}$ for $f\in \mathbb R\setminus \{0\}$.

We now have one last case.

\subsection{$G_3$.}\label{SubsectionT1}
The only nontrivial extension of  $\mathbb R\ltimes_c \mathbb R$ by $\mathbb R$ that arises comes equipped with the action as described above in~\ref{Subsectionfinal} with $c=d$. As there, the mapping $(x,y,z)\mapsto (|c|x,y,z)$
provides an ordered group isomorphism that allows us to assume that either $c=d=1$ or $c=d=-1$.
We will refer to the respective ordered groups as $T_1=G_{3}$ and $T_{-1}$.

By the same argument given in the proof of Lemma \ref{P:2.4.3}, any nontrivial extension of $\mathbb R\ltimes_c \mathbb R$ by $\mathbb R$ must be (definably) isomorphic as an ordered group to either $G_3=T_1$ or to $T_{-1}$. Moreover, it is evident that these groups are not isomorphic as ordered groups since the action of the quotient of the group by its two dimensional normal convex (abelian) subgroup on this subgroup is increasing in $T_1$ and decreasing $T_{-1}$.

\section{Summary and relations with Real Lie Groups}\label{Summary}

Let $\mathcal R$ be an o-minimal expansion of an ordered field satisfying the hypotheses
$(*)$ and $(**)$, namely that all o-minimal groups
are definably isomorphic and (as stated above for expansions of the real field) that every abelian torsion free two dimensional definable group is isomorphic to the direct product of its additive group with itself.

We here recapitulate the results we have proved, which for concreteness, we state assuming that $\mathcal R$ is an o-minimal expansion of the real field.

\begin{summary}\label{T:3*}
The supersolvable definable groups in $\mathcal R$ that are definably (topologically) homeomorphic to $\mathbb R^3$ are,
modulo definable group isomorphism, $\left(\mathbb{R}^{3},+\right)$ ,
$\left(\mathcal{G}_{\textrm{Heis}},\cdot\right)$,
$\left(\left(\mathbb{R}\rtimes\mathbb{R}\right)\times\mathbb{R},\cdot\right)$, $G_3$,
and $\mathbb{R}^{2}\rtimes_{\tau_{c}}\mathbb{R}$ where
$\tau_{c}\left(z\right)\left(x,y\right)=\left(x{e^{cz}},
y{e^{z}}\right)$, with $c\not=0$, is a nontrivial action  of
$\left(\mathbb{R},+\right)$ on $\left(\mathbb{R}^{2},+\right)$.
\end{summary}

For 2-dimensional ordered groups definable in $\mathcal R$ we have
\begin{summary}\label{2-dimensional}
Let $G$ be a $\mathcal R$-definable 2-dimensional ordered group. Then
$G$ is isomorphic to $\mathbb R^2:=\mathbb R\times \mathbb R$, to
$\mathbb R \rtimes_1 \mathbb R$, or $\mathbb R \rtimes_{-1} \mathbb R$, each
ordered with the reverse lexicographical order (that is, the second coordinate ordered first).
\end{summary}

Before stating our result for ordered groups of dimension~3 definable in $\mathcal R$ we introduce some notation. For
$\mathbb R^3$ written as triples $(x,y,z)$, by $x>>y>>z$ we mean the lexicographic order with the $x$-coordinate first, then the $y$-coordinate, and lastly the $z$-coordinate. Likewise, $y>>x>>z$ is the lexicographic order with the $y$-coordinate first, then the $x$-coordinate, and lastly the $z$-coordinate, and similarly for the other permutations of the variables.

\begin{summary}\label{3-dimensional}
Let $G$ be a $\mathcal R$-definable ordered group of dimension 3.
Then $G$ is definably isomorphic as an ordered group to exactly one of the following:
\begin{itemize}
\item[i.] $\left(\mathbb{R}^{3},+\right)$ with lexicographical order.

\item[ii.]
With the notation from~\ref{8.2}, either $E_1$ or $E_{-1}$, ordered lexicographically (first the $x$-coordinate, then $y$, and then $z$). The groups represented in this class are the nontrivial extensions of $\mathbb R^2$ by
$\mathbb R$ with trivial action.

\item[iii.] $\left(\left(\mathbb{R}\rtimes\mathbb{R}\right)\times\mathbb{R},\cdot\right)$. If the presentation is given by $\{(x,y,z)\mid x,y,z\in \mathbb R^3\}$, then
the order is lexicographic, with either $y>>z>>x$, $z>>y>>x$, or $y>>x>>z$. The first two cases are from~\ref{8.3}; the last from~\ref{8.4}.

\item[iv.] $\mathbb{R}^{2}\rtimes_{\tau}\mathbb{R}$, with universe $\{(x,y,z)\mid x,y,z\in \mathbb R^3\}$ and
$\tau$ given by either $\tau\left(z\right)\left(x,y\right)=\left(e^{z}x,
e^{d}y\right)$ or $\tau\left(z\right)\left(x,y\right)=\left(e^{-z}x,e^{dz}y\right)$ for some $d\in \mathbb R\setminus 0$, ordered by $z>>x>>y$.

\item[v.] $T_{1}$ or $T_{-1}$ ,with the notation used in  \ref{SubsectionT1}, ordered lexicographically.
\end{itemize}
\end{summary}

\bigskip

We now turn to the real Lie group context. It is known that any compact Lie group is isomorphic
to a real algebraic subgroup of $GL(n, \mathbb R)$ (see, \cite{Kn}, e.g.). Such a group thus is interpretable in an o-minimal structure (in fact, definable in the real field) and is therefore analyzable by results in the o-minimal setting. It is also
not difficult to find examples of Lie groups---over either the complex or real field---that are not definable in
an o-minimal structure (e.g., the group of rigid---norm preserving---automorphisms of the plane).
However, our context rules out all of the examples known to us: any real Lie group with an order compatible with the group structure must be topologically homeomorphic to a cartesian product of $\mathbb R$. We do not know
whether or not any ordered real Lie group, where the order is given by smooth functions is definably isomorphic to one of the groups described above. In fact, we do not know whether or not a Lie group with an order that is compatible with the group structure is supersolvable.

We do know the following.

\begin{corollary}\label{algebraicLie}
Let $G$ be a algebraic real Lie ordered group (in the sense that the group operations and the order are definable in the
real ordered field structure, i.e., by polynomials and the natural order on $\mathbb R$) of dimension~2 or~3. Then $G$ is isomorphic (as an ordered Lie group)
to one of the groups described above in \ref{2-dimensional} and \ref{3-dimensional}. Even more, the isomorphism is definable in
$\mathbb R_{\it Pfaff}$, and hence is given by equations that are solutions to a Pfaffian system over the real field (Pfaffian equations). In particular the isomorphism is analytic.
\par

The same is true for any real Lie ordered group of dimension~2 or~3 for which the group operation and the order are definable via  Pfaffian equations.
\par

Conversely, all of the groups described in \ref{2-dimensional} and \ref{3-dimensional} are definable by Pfaffian equations. However, the groups $T_1=G_3$ and 
$T_{-1}$ are not algebraic, so these are Lie groups which are orderable by continuous relations that are not algebraic.
\end{corollary}

We believe that something along the lines of Corollary~\ref{algebraicLie} should be true for all Lie ordered groups, since refuting this would require adding a Lie (ordered) group structure that violates o-minimality of the real field, without having the universe be discrete (since an ordered Lie group must be topologically homeomorphic to $\mathbb R^n$ for some $n$). While this seems unlikely, we do not know how to continue.

\bibliographystyle{abbrv}
\bibliography{decomposable}

\end{document}